\documentclass[11pt]{article}
\usepackage{latexsym,color,amssymb,amsthm,fullpage,enumitem}
\usepackage{graphicx}

\usepackage[left=1in, right=1in, top=1in, bottom=1in]{geometry}

\newcommand{\ignore}[1]{}

\newtheorem{theorem}{Theorem}[section]
\newtheorem{lemma}[theorem]{Lemma}

\newtheorem{proposition}[theorem]{Proposition}
\newtheorem{claim}[theorem]{Claim}

\newtheorem{corollary}[theorem]{Corollary}

\def\eps{{\epsilon}}
\def\conv{{\sf conv}}

\def\A{{\cal A}}
\def\Y{{\cal Y}}

\def\F{{\cal F}}
\def\L{{\cal L}}

\def\K{{\cal K}}

\def\W{{\cal W}}
\def\R{{\mathcal R}}

\def\reals{{\mathbb R}}

\def\G{{\cal G}}

\def\E{{\cal{E}}}

\begin{document}

\begin{titlepage}

\title{An Improved Bound for Weak Epsilon-Nets in the Plane\footnote{
The project leading to this application has received funding from European Research Council (ERC)
under the European Unions Horizon 2020 research and innovation programme under grant agreement No. 678765. An extended abstract of this work has appeared in {Proceedings of the 59th Annual IEEE Symposium on 
Foundations of Computer Science}, 
October 7-9, 2018, Paris, France.}}

\author{
Natan Rubin\thanks{Email: {\tt rubinnat.ac@gmail.com}. Ben Gurion University of the Negev, Beer-Sheba. Also supported
by grant 1452/15 from Israel Science Foundation and by grant 2014384 from the U.S.-Israeli Binational Science Foundation.} }

\maketitle

\begin{abstract}

We show that for any finite point set $P$ in the plane and $\eps>0$ there exist $\displaystyle O\left(\frac{1}{\eps^{3/2+\gamma}}\right)$ points in $\reals^2$, for arbitrary small $\gamma>0$, that pierce every convex set $K$ with $|K\cap P|\geq \eps |P|$. This is the first improvement of the bound of $\displaystyle O\left(\frac{1}{\eps^2}\right)$ that was obtained in 1992 by Alon, B\'{a}r\'{a}ny, F\"{u}redi and  Kleitman for general point sets in the plane.
\end{abstract}

\maketitle


\end{titlepage}

\section{Introduction} \label{sec:intro}
\paragraph{Transversals and $\eps$-nets.} Given a family $\K$ of geometric ranges in $\reals^d$ (e.g., lines, triangles, or convex sets),
we say that $Q\subset \reals^d$ is a transversal to $\K$ (or $Q$ pierces $\K$) if each $K\in \K$ is pierced by at least one point of $Q$.
Given an underlying set $P$ of $n$ points, we say that a range $K\in \K$ is {\it $\eps$-heavy} if $|P\cap K|\geq \eps n$.
We say that $Q$ is an {\it $\eps$-net} for $\K$ if it pierces every $\eps$-heavy range in $\K$.
We say that an $\eps$-net for $\K$ is a {\it strong $\eps$-net} if $Q\subset P$, that is, the points of the net are drawn from the underlying point set $P$. Otherwise (i.e., if $Q$ includes additional points outside $P$), we say that $Q$ is a {\it weak \it $\eps$-net}.

The study of $\eps$-nets was initiated by Vapnik and Chervonenkis \cite{VC1971}, in the context of statistical learning theory.
Following a seminal paper of Haussler and Welzl \cite{HW87}, $\eps$-nets play a central role in discrete and computational geometry \cite{Trends}.
For example, bounds on $\eps$-nets determine the performance of the best-known algorithms for minimum hitting set/set cover problem in geometric hypergraphs \cite{NetsRectangles,BroGood,ClaVar,Even}, and the transversal numbers of families of convex sets \cite{AlonKalai,AKMM,AlonKleitman,Shakhar}. 

Informally, the cardinality of the smallest possible $\eps$-net for the range set $\K$ determines the integrality gap of the corresponding transversal problem -- the ratio between (1) the size of the smallest possible transversal $Q$ to $\K$ and (2) the weight of the ``lightest" possible fractional transversal to $\K$ \cite{AlonKleitman,AlonKalai,Even}.

Haussler and Welzl \cite{HW87} proved in 1986 the existence of strong $\eps$-nets of cardinality $O\left(\frac{1}{\eps}\log\frac{1}{\eps}\right)$ for families of so called semi-algebraic ranges of bounded description complexity in $d$-space\footnote{Namely, each of these sets is the locus of all the points that satisfy a given Boolean combination of a bounded number of algebraic equalities and inequalities of bounded degree in the Euclidean coordinates $x_1,\ldots,x_d$ \cite{SemiAlgebraic}.}, for a fixed $d>0$ (e.g., lines, boxes, spheres, halfspaces, or simplices), by observing that their induced hypergraphs have a bounded Vapnik-Chervonenkis dimension (so called {\it VC-dimension}).\footnote{The constant hidden within the $O(\cdot)$-notation is specific to the family of geometric ranges under consideration, and is proportional to the VC-dimension of the induced hypergraph.}
While the bound is generally tight for set systems with a bounded VC-dimension \cite{KPW90}, better constructions are known for several special families of geometric ranges, including tight bounds for discs in $\reals^2$, halfplanes in $\reals^2$ and halfspaces in $\reals^3$ \cite{ClaVar,KPW90,MSW90}, and rectangles in $\reals^2$ and boxes in $\reals^3$ \cite{NetsRectangles,PachTardos}. We refer the reader to a recent state-of-the-art survey \cite{HandbookNets} for the best known bounds.

In particular, it had long been conjectured that all the ``natural" geometric instances, that involve simply-shaped geometric ranges in a fixed-dimensional Euclidean space $\reals^d$, admit a strong $\eps$-net 
of cardinality $O(1/\eps)$ (where the constant of proportionality depends on the VC-dimension).
The conjecture was refuted in 2010 by Alon \cite{Alon} who used a density version of  Hales-Jewett Theorem \cite{HJ} to show that some families of geometric ranges (e.g., lines in the Euclidean plane) require an $\eps$-net whose cardinality is larger than $O(1/\eps)$. Pach and Tardos \cite{PachTardos} subsequently demonstrated that the multiplicative term $\Theta\left(\log 1/\eps\right)$ is necessary for strong $\eps$-nets with respect to halfspace ranges in dimension higher than $3$.

\paragraph{Weak $\eps$-nets for convex sets.} In sharp contrast to the case of simply-shaped ranges, no constructions of small-size strong $\eps$-nets exist for general families of convex sets in $\reals^d$, for $d\geq 2$.  For example, given an underlying set of $n$ points in convex position in $\reals^2$, any strong $\eps$-net with respect to convex ranges must include at least $n-\eps n$ of the points. Informally, this  phenomenon can be attributed to the fact that 
the VC-dimension of a geometric set system is closely related to the {\it description complexity} of the underlying ranges, and it is unbounded for general convex sets.
Nevertheless, B\'{a}r\'{a}ny, F\"{u}redi and Lov\'{a}sz \cite{BFL} observed in 1990 that families of convex sets in $\reals^2$ still admit weak $\eps$-nets of cardinality $O(\eps^{-1026})$.
Alon, B\'{a}r\'{a}ny, F\"{u}redi, and Kleitman \cite{AlonSelections} were the first to show in 1992 that families of convex sets in any dimension $d\geq 1$ admit weak $\eps$-nets whose cardinality is bounded in terms of $1/\eps$ and $d$. The subsequent study and application of weak $\eps$-nets bear strong relations to convex geometry, including Helly-type, Centerpoint and Selection Theorems; see \cite[Sections 8 -- 10]{JirkaBook} for a comprehensive introduction. 

\paragraph{Weak $\eps$-nets and the Hadwiger-Debrunner Problem.} Alon and Kleitman \cite{AlonKleitman} used the boundedness of weak $\eps$-nets to confirm a long-standing {\it $(p,q)$-conjecture} by Hadwiger and Debrunner \cite{HD}. To this end, we say that a family $\K$ of convex sets satisfies the {\it $(p,q)$-property}  if any its $p$-size subfamily $\K'\subset \K$ contains a $q$-size subset $\K''\subset \K$ with a non-empty common intersection $\bigcap \K''\neq \emptyset$.
Hadwiger and Debrunner conjectured (and Alon and Kleitman showed) that for every positive integers $p,q$ and $d$ that satisfy $p\geq q\geq d+1$, there exists an integer $C_d(p,q)<\infty$ so that the following statement holds: Any family $\K$ of convex sets in $\reals^d$ with the $(p,q)$-property admits a transversal by at most $C_d(p,q)$ points.\footnote{The celebrated Helly Theorem yields a transversal by a {\it single} point in the case $p=q=d+1$ whenever $|\K|\geq d+1$.}
Showing good quantitative estimates for the Hadwiger-Debrunner numbers $C_d(p,q)$ requires tight asymptotic bounds for weak $\eps$-nets; see the latest study by Keller, Smorodinsky and Tardos \cite{Shakhar}, and the concluding discussion in Section \ref{Sec:Conclude}.

\medskip
Very recently, lower bounds for several of the above questions -- including strong and weak $\eps$-nets with respect to line ranges \cite{BS18}, and the 2-dimensional Hadwiger-Debrunner numbers $C_2(p,q)$ \cite{KS18} -- were improved using the novel combinatorial machinery of hypergraph containers \cite{BMS,SaxTh}.

\paragraph{Bounds on weak $\eps$-nets.} For any $\eps>0$ and $d\geq 0$, let $f_d(\eps)$ be the smallest number $f>0$ so that, for any underlying finite point set $P$, one can pierce all the $\eps$-heavy convex sets
using only $f$ points in $\reals^d$. 
It is
an outstanding open problem in Discrete and Computational geometry to determine the true asymptotic behaviour of $f_d(\eps)$ in dimensions $d\geq 2$. As Alon, Kalai, Matou\v{s}ek, and Meshulam noted in 2001: ``{\it Finding the correct estimates for weak $\eps$-nets is, in our opinion, one of the truly important open problems in combinatorial geometry"} \cite{AKMM}.

\medskip
Alon {\it et al.} \cite{AlonSelections} (see also \cite{AlonKleitman}) used Tverberg-type results to show that $f_d(\eps)=O(1/\eps^{d+1-1/\beta_d})$ (where $0<\beta_d<1$ is a selection ratio which is fixed for every $d$), and $f_2(\eps)=O\left(1/\eps^2\right)$.\footnote{An outline of the planar $f_2(\eps)=O\left(1/\eps^2\right)$ bound can be found in a popular textbook by Chazelle \cite{ChazelleBook}.}
The bound in higher dimensions $d\geq 3$ has been subsequently improved in 1993 by Chazelle {\it et al.} \cite{Chazelle} to roughly $\tilde{O}\left(\frac{1}{\eps^d}\right)$ (where $\tilde{O}(\cdot)$-notation hides multiplicative factors that are polylogarithmic in $\log 1/\eps$). Though the latter construction was somewhat simplified in 2004 by Matou\v{s}ek and Wagner \cite{MatWag04} using simplicial partitions with low hyperplane-crossing number \cite{PartitionTrees}, no improvements in the upper bound for general families of convex sets and arbitrary finite point sets occurred for the last 25 years,  in any dimension $d\geq 2$.

In view of the best known lower bound of $\Omega\left(\frac{1}{\eps} \log^{d-1}\left(\frac{1}{\eps}\right)\right)$ for $f_d(\eps)$ due to Bukh, Matou\v{s}ek and Nivasch \cite{Staircase},  it still remains to settle whether the asymptotic behaviour of this quantity substantially deviates from the long-known ``almost-$(1/\eps)$" bounds on strong $\eps$-nets (e.g., for lines and triangles in $\reals^2$ or simplices in $\reals^d$)? 

The only interesting instances in which the gap has been essentially closed, involve special point sets \cite{Chazelle,Sphere,AlonChains}.
For example, Alon {\it et al.} \cite{AlonChains} showed in 2008 that any finite point set in {\it a convex position in $\reals^2$} allows for a weak $\eps$-net of cardinality $\displaystyle O\left(\alpha(\eps)/\eps\right)$ with respect to convex sets, where $\alpha(\cdot)$ denotes the inverse Ackerman function.

\paragraph{Our result and organization.} We provide the first improvement\footnote{A follow-up study of the author \cite{STOC} establishes a similar improvement in all dimensions $d\geq 3$. Its remarkable that the higher-dimensional argument does not yield Theorem \ref{Thm:Main} for $d=2$; see Section \ref{Sec:Conclude} for a brief comparison.} of the general bound in $\reals^2$.

\begin{theorem}\label{Thm:Main}
We have
$$
f_2(\eps)=O\left(\displaystyle\frac{1}{\eps^{3/2+\gamma}}\right),
$$ 
for any $\gamma>0$.

That is, for any underlying set $P$ of $n$ points in $\reals^2$, and any $\eps>0$, one can construct a weak $\eps$-net with respect to convex sets whose cardinality is $O\left(\displaystyle\frac{1}{\eps^{3/2+\gamma}}\right)
$; here $\gamma>0$ is an arbitrary small constant which does not depend on $\eps$.\footnote{We do not seek to optimize the implicit constant of proportionality within $O(\cdot)$, which heavily depends on $\gamma>0$. The particular recurrence scheme that we establish for $f_2(\eps)$ in Section \ref{Subsec:WrapUp} results in a constant that is super-exponential in $1/\gamma$.}
\end{theorem}

The main underlying idea of our proof of Theorem \ref{Thm:Main} is that one can find a small auxiliary net $Q'$, of $o\left(1/\eps^2\right)$ points, with the property that every $\eps$-heavy convex set $K$ that is missed by $Q'$ is ``line-like" -- the subset $P\cap K$ is largely determined by an edge that connects some pair of its points; in the sequel, these sets are called { narrow}. (As a matter of fact, the number of such ``proxy" edges for a narrow set $K$ is close to ${\eps n \choose 2}$.)  These narrow sets are pierced by a careful adaptation of the decomposition paradigm that was previously used to bound the number of point-line incidences in the plane \cite{SzT,ManyCells}.

\medskip
The rest of the paper is organized as follows:

\smallskip
In Section \ref{Sec:Prelim} we provide a comprehensive overview of our approach, lay down the recursive framework, and establish several basic properties that are used throughout the proof of Theorem \ref{Thm:Main}. 

In Section \ref{Sec:Main} we use the recursive framework of Section \ref{Sec:Prelim} to give a constructive proof of Theorem \ref{Thm:Main}. The eventual net combines the following elementary ingredients: (1) vertices of certain trapezoidal decompositions of $\reals^2$, (2) $1$-dimensional $\hat{\eps}$-nets, for $\hat{\eps}=\omega(\eps^2)$, which are constructed within few vertical lines with respect to carefully chosen point sets, and (3) strong $\hat{\eps}$-nets with respect to triangles in $\reals^2$.

In Section \ref{Sec:Conclude} we briefly summarize the properties of our construction and survey the future lines of work.

\section{Preliminaries}\label{Sec:Prelim}

\subsection{Proof outline}\label{Subsec:ProofOutline}

We briefly outline the main ideas behind our proof of Theorem \ref{Thm:Main}. We begin by sketching the $O\left(1/\eps^2\right)$ planar construction of Alon {\it et al.} \cite{AlonSelections} (or, rather, its more comprehensive presentation by Chazelle \cite{ChazelleBook}).

\paragraph{The quadratic construction.}  Refer to Figure \ref{Fig:Overview} (left).
We split the underlying point set $P$ by a vertical median line $L$ into subsets $P^-$ and $P^+$ (of cardinality $n/2$ each), and recursively construct a weak $(4\eps/3)$-net with respect to each of these sets. 
Let $K$ be an $\eps$-heavy convex set; in what follows, we assign to each such set $K$ a unique subset $P_K\subseteq P\cap K$ of exactly $\lceil \eps n\rceil$ points. If at least $3\eps n/4$ points of $P_K$ lie to the same side of $L$, we pierce $K$ by one of the auxiliary $(4\eps/3)$-nets. Otherwise, the points of $P_K$ span at least $\eps^2n^2/16$ edges that cross $K\cap L$, so we can pierce $K$ by adding to our net each $(\eps^2n^2/16)$-th crossing point of $L$ with the edges of ${P\choose 2}$.\footnote{In the sequel we use ${A\choose 2}$ to denote the complete set of edges spanned by a (finite) point set $A\subset \reals^2$.} 

The above argument yields a recurrence of the form $f_2(\eps)\leq 2f_2\left(4\eps/3\right)+16/\eps^2$ which bottoms out when $\eps$ surpasses $1$ (in which case we use the trivial bound $f(\eps)\leq 1$ for all $\eps\geq 1$).

\begin{figure}
    \begin{center}
    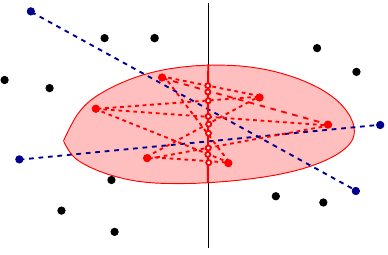\hspace{2cm}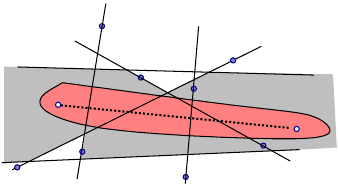
        \caption{\small Left: Constructing the net of cardinality $O(1/\eps^2)$. If the points of $P_K$ are well distributed between $P^-$ and $P^+$, the intercept $K\cap L$ is crossed by $\Theta(\eps^2n^2)$ edges of ${P_K\choose 2}$. Notice that the intercept $K\cap L$ can be crossed by many edges outside ${P_K\choose 2}$. Right: Our decomposition of $\reals^2$ uses cells of the arrangement of certain lines which are sampled from among the lines spanned by $P$. The depicted set $K$ is narrow -- its zone is also the zone of the principal edge $pq$.}
    \label{Fig:Overview}
   \end{center}
   
\end{figure}

Notice that the above approach immediately yields a net of size $o(1/\eps^2)$ for sets $K$ that fall into one of the following favourable categories:

\begin{enumerate}

\item  The interval $K\cap L$ is crossed by more than $\Theta(\eps^2 n^2)$ edges of ${P\choose 2}$, with either one or both of their endpoints lying outside $K$.

For example, we need only $1/\delta=o\left(1/\eps^2\right)$ points to pierce such sets $K$ whose cross-sections $K\cap L$ contain at least $\delta n^2=\omega\left(\eps^2n^2\right)$ intersection points of $L$ with the edges of ${P\choose 2}$.

\item At least a fixed fraction of the $\Omega\left(\eps^2n^2\right)$ edges spanned by $P_K$ belong to a relatively sparse subset $\Pi\subset{P\choose 2}$ of cardinality $m=o(n^2)$. This subset $\Pi$ is carefully constructed in advance and does not depend on the choice of $K$. 

This too leads to a net of size $O\left(m/(\eps^2 n^2)\right)=o\left(1/\eps^2\right)$ provided that a large fraction of these edges of $\Pi$ end up crossing $L$. (In other words, the endpoints of these edges must be sufficiently spread between the halfplanes of $\reals^2\setminus L$.)

\end{enumerate}

\paragraph{Decomposing $\reals^2$.} To force at least one of the above favourable scenarios, we devise a randomized decomposition of $\reals^2$ and $P$. Rather than using a single line to split $\reals^2$ into halfplanes, we use a subset $\R$ of $r=o\left(1/\eps\right)$ lines that are chosen at random from among the lines that support the edges of ${P\choose 2}$, and consider their entire {\it arrangement} $\A(\R)$ -- the decomposition of $\reals^2\setminus \bigcup \R$ into open $2$-dimensional faces. (See Section \ref{Subsec:Essentials} for the precise definition of an arrangement, and its essential properties.)
We use the ${r\choose 2}=o\left(1/\eps^2\right)$ vertices of $\A(\R)$ to construct a small-size point set $Q$ with the following property:
Every convex set $K$ that is {\it not} pierced by $Q$ must demonstrate a ``line-like" behaviour with respect to $\A(\R)$ -- its zone (namely, the 2-faces intersected by $K$) must be contained, to a large extent, in the zone of a single edge $pq\in {P_K\choose 2}$. See Figure \ref{Fig:Overview} (right). In what follows, we will refer to such convex sets as {\it narrow}. Though such a ``proxy" edge $pq$ can be selected from ${P_K\choose 2}$ in (almost) $\Theta(\eps^2n^2)$ ways, we will assign a unique ``proxy" edge $pq\in {P_K\choose 2}$ to every narrow convex set $K$, and refer to that edge $pq$ as the {\it principal edge} of $K$.

\paragraph{Representing narrow convex sets by edges.} The fundamental difficulty of representing and manipulating convex sets (as opposed to lines, segments, simplices, and other simply-shaped geometric objects) is that they can cut the underlying point set $P$ into exponentially many subsets $P\cap K$, so the standard divide-and-conquer schemes  \cite{ManyCells} hardly apply in this setting. Fortunately, every narrow convex set $K$ can be largely described by its principal edge $pq\in {P_K\choose 2}$. (For example, $K$ cannot include points outside the respective zone of $pq$.)

\begin{figure}
    \begin{center}
        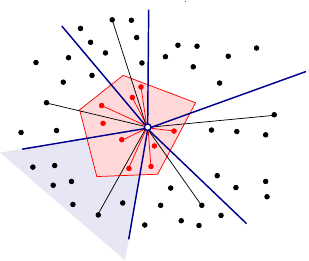\hspace{2cm}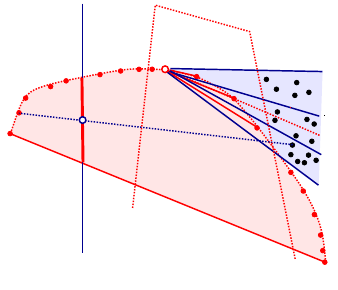
        \caption{\small Left: We partition the plane into $z=O(1/\eps)$ sectors $\W_j(u)$, each containing roughly $\eps n$ outgoing edges $uv$, and an average amount of $O(\eps n/r^2)$ outgoing short edges. Right: The point $u$ with the outgoing short edges that are ``parallel" to the principal edge $pq$, and whose supporting lines are roughly tangent to $K$. In Case 1, the $\Omega(\eps n/r)$ outgoing short edges of $u$ within $\Delta\cap K$ occupy multiple sectors $\W_j(u)$ which are almost tangent to $K$. This yields $\omega(\eps^2n^2)$ segments that cross the intercept $K\cap L$.}
        \label{Fig:Sectors}
    \end{center}
    \vspace{-0.3cm}
\end{figure}

\paragraph{From narrowness to expansion.} The main geometric phenomenon behind our choice of the sparse (i.e., non-dense) subset $\Pi\subset {P\choose 2}$ is that the ``expected" rate of expansion
of $P_K$ within the arrangement $\A(\R)$ from a point $u\in P_K$, for a narrow convex set $K$, is generally lower than that of the entire set $P$ from that same point.\footnote{To this end, we define the pseudo-distance between a pair of points $u,v\in \reals^2$ as the number of lines in $\R$ that are crossed by the open segment $uv$; see \cite[Section 2.8]{ChazelleBook} and \cite{ChazelleWelzl}. For a finite set $A\subset \reals^2$, and a point $u\in A$, we examine the ``expected" order of magnitude of the volume $|A\cap D(u,\theta)|$ of the disc $D(u,\theta)$ as a function of $\theta\geq 0$. Clearly, this informal notion is related to the more standard concepts of doubling dimension \cite{Doubling} and graph expansion \cite{Expanders}.}

To illustrate this behaviour, assume first that the points of $P$ are evenly distributed among the cells of $\A(\R)$, so each cell contains roughly $n/r^2$ points.  We say that an edge $uv\in {P\choose 2}$ is {\it short} if both of its endpoints lie in the same cell of $\A(\R)$.

For each point $u$ of $P$ we partition the plane into $z=\Theta\left(\frac{1}{\eps}\right)$ sectors
 $\W_1(u),\W_2(u),\ldots,\W_z(u)$ so that each sector encompasses $\Theta(\eps n)$ outgoing edges $uv\in {P\choose 2}$; see Figure \ref{Fig:Sectors} (left). To pierce a narrow convex set $K$ whose zone in $\A(\R)$ is traced by an edge $pq\in {P_K\choose 2}$,  we combine the following key observations:

\begin{enumerate}

\item[i.] For an {\it average point $u$} in $P_K$, its cell $\Delta$ contains at least $\eps n/(r+1)$ points of $P_K$, which are connected to $u$ by short edges (because $K$ crosses at most $r+1$ cells of $\A(\R)$). 

\item[ii.]  For an {\it average edge $uv\in {P\choose 2}$}, the respective sector $\W_j(u)$ contains only $O\left(\eps n/r^2\right)$ short edges.
\end{enumerate}

We further guarantee that the points of $P_K$ are in a sufficiently convex position,
and are substantially distributed in the zone of $K$: The former property is enforced by using a strong $\hat{\eps}$-net \cite{HW87}, with $\hat{\eps}=\Theta(\eps/r)$, to eliminate the forbidden convex sets $K$, whereas the latter condition is enforced using a suitably amplified version of the prior line-splitting argument.
Thus, for $\Omega(\eps n)$ choices of $u\in P_K$, we can assume that both endpoints of the principal edge $pq\in {P\choose 2}$ of $K$ lie outside the cell $\Delta$ of $u$, and at least half of the $\Omega(\eps n/r)$ points $v\in P_K\setminus \{p\}$ within $\Delta$ (which exist by property (i)) lie to the same side of $pq$ as $u$. By the near convex position of $P_K$, most lines spanned by such short edges $uv$ within $\Delta$ are roughly tangent to the convex hull of $P_K$; see Figure \ref{Fig:Sectors} (right). (In particular, the four points $p,u,v,q$ form a convex quadrilateral.) 

Assume with no loss of generality that at least half of the above short edges $uv$ are parallel to $pq$, in the sense that the four points $p,u,v,q$ appear in this order along their convex hull. Since an average sector $\W_j(u)$ contains only $O(\eps n/r)$ such edges, we interpolate between the following scenarios.

\paragraph{Case 1.} The wedge spanned by the above $\Omega(\eps n/r)$ short edges $uv\in {P_K\choose 2}$ (along with $uq$) occupies $r$ ``average" sectors $\W_{j}(u),\W_{j+1}(u),\ldots, \W_{j+r}(u)$, which are almost tangent to $K$. We show that the points of $P$ within $\W_j(u)\cup \W_{j+1}(u)\cup\ldots\cup\W_{j+r}(u)$ yield $r\eps^2n^2$ adjacent edges that cross the intercept $K\cap L$ of $K$ with the ``middle" vertical line $L$ that we use to split the points of $P$. (Again, see Figure \ref{Fig:Sectors} (right).) Hence, the intersection of $K\cap L$ is relatively ``thick", so we can pierce such sets using $O\left(1/(r\eps^2)\right)$ points.

\paragraph{Case 2.} The previous scenario does not occur. Using the near-convexity of $P_K$, we find $\Omega(\eps n)$ outgoing edges of $u$ within ${P_K\choose 2}$ that are parallel to $pq$ in the above sense and occupy a constant number of {\it rich} sectors $\W_j(u)$ with at least $\Omega(\eps n/r)$ short edges. 

Property (ii) implies that there exist $O\left(1/(r\eps)\right)$ rich sectors $\W_j(u)$, which encompass a total of $O(n/r)$ edges that emanate from $u$.
To pierce such convex sets $K$ that fall into Case 2, we define our sparse set $\Pi\subset {P\choose 2}$ as the set of edges $uv$ which lie in rich sectors $\W_j(u)$, $\W_{j'}(v)$ (for at least one of the respective endpoints $u$ or $v$). Note that the construction does not depend on the convex set $K$. It is easy to check that $P_K$ spans at least $\Omega(\eps^2n^2)$ such edges within $\Pi$, and sufficiently many of these edges must cross $L$. Hence, $K$ falls into the second favourable case.

\paragraph{The vertical decomposition.} Since the actual distribution of $P$ in $\A(\R)$ is not necessarily uniform, we subdivide the cells of $\A(\R)$ into a total of $O(r^2)$ more homogeneous trapezoidal cells, so that each cell contains at most $n/r^2$ points of $P$.\footnote{A similar decomposition was used, e.g., by Clarkson {\it et al.} \cite{ManyCells} to tackle the closely related problem of bounding generalized point-line incidences (e.g., incidences between points and unit circles, or incidences between lines and certain cells of their arrangement); the relation between the two problems is briefly discussed in the concluding Section \ref{Sec:Conclude}.} To adapt the preceding expansion argument to the faces of the resulting decomposition $\Sigma$, we extend the notion of narrowness to $\Sigma$ and guarantee that every narrow convex set $K$ crosses only a small fraction of the faces in $\Sigma$.
More specifically, the (strong) Epsilon Net Theorem implies that any trapezoidal cell $\tau$ is crossed by $O(n^2\log r/r)$ of the lines that support the edges of ${P\choose 2}$,\footnote{In other words, $\Sigma$ is a $\Theta(r/\log r)$-cutting \cite{Cuttings} of $\reals^2$ with respect to these lines.} so an average edge of ${P\choose 2}$ crosses only $O(r\log r)$ trapezoidal cells of $\Sigma$. 
As a result, a ``typical" narrow convex set $K$ (whose zone in $\Sigma$ can ``read off" from its principal edge $pq$ within ${P_K\choose 2}$) crosses relatively few faces of $\Sigma$; in other words, $K$ has a {\it low crossing number} with respect to $\Sigma$.

The ``exceptional" convex sets $K$, which cross too many faces of $\Sigma$, are dispatched separately using that, for $\Theta(\eps^2n^2)$ of their edges, their supporting lines cross too many cells $\Sigma$ and, thereby, belong to another sparse subset of ${P\choose 2}$.

\paragraph{Discussion.} 
Our trapezoidal decomposition $\Sigma$  of $\reals^2$ overly resembles the first step of the proof of the Simplicial Partition Theorem of Matou\v{s}ek \cite{PartitionTrees} in dimension $d=2$, which provides $s=O(r^2)$ triangles $\Delta_1,\ldots,\Delta_s$ so that each triangle $\Delta_i$ contains $\Theta(n/s)$ points, and any line in $\reals^2$ crosses $O(\sqrt{s})=O(r)$ of these triangles. 

It is instructive to compare our approach to the partition-based technique of Matou\v{s}ek and Wagner \cite{MatWag04}, which directly uses the above theorem in $\reals^d$ to re-establish the near-$1/\eps^d$ bounds of Alon {\it et al.} \cite{AlonSelections} (in $\reals^2$) and Chazelle {\it et al.} \cite{Chazelle} (in any dimension $d\geq 2$), via a simple recursion on the point set and the parameter $\eps$.

Notice that the triangles $\Delta_i$ in the Simplicial Partition Theorem, for $1\leq i\leq s$, cannot be related to particular cells of any single arrangement of lines. 
To enforce a low crossing number among the convex sets $K$ with respect to the partition $\{\Delta_1,\ldots,\Delta_s\}$,  Matou\v{s}ek and Wagner pick a point $p_i$ in each triangle $\Delta_i$ and pierce the ``exceptional" sets by an auxiliary net of $O\left(s^{O\left(d^2\right)}\right)$ centerpoints which are obtained via Rado's Centerpoint Theorem \cite{JirkaBook} for all the possible subsets of $\{p_i\mid 1\leq i\leq s\}$. 
Hence, the cardinality of their net heavily depends on the size  $s$ of the partition. Unfortunately, their simplicial partition does not quite suit our analysis, which needs a relatively large number of cells to achieve a substantial improvement over the $O\left(1/\eps^2\right)$ bound.

\subsection{The recursive framework.}\label{Subsec:RecursiveFramewk}
We refine the notation of Section \ref{sec:intro} and lay down the formal framework in which our construction  and its analysis are cast.


\medskip
\noindent{\bf Definition.} For a finite point set $P$ in $\reals^2$ and $\eps>0$, let $\K(P,\eps)$ denote the family of all the $\eps$-heavy convex sets with respect to $P$.
We then say that $Q\subset \reals^2$ is {\it a weak $\eps$-net} for a family of convex sets $\G$ in $\reals^2$ if it pierces every set in $\G\cap \K(P,\eps)$.

\medskip
If the parameter $\eps$ is fixed, we can assume that each set in $\K$ is $\eps$-heavy, so $Q$ is simply a point transversal to $\K$.
Note also that every weak $\eps$-net with respect to $P$ is, in particular, a weak $\eps$-net with respect to any subfamily $\K$ of convex sets in $\reals^2$.

\medskip
Notice that the previous constructions \cite{AlonChains,Chazelle,MatWag04} employed recurrence schemes
in which every problem instance $(P,\eps)$ was defined over a finite point set $P$, and sought to pierce each $\eps$-heavy convex set $K\in \K(P,\eps)$ using the smallest possible number of points. This goal was achieved in a divide-and-conquer fashion, by tackling a number of simpler sub-instances $(P',\eps')$ with a smaller point set $P'\subset P$ and a larger parameter $\eps'>\eps$.
 
To amplify our sub-quadratic bound on $f_2(\eps)$, we employ a somewhat more refined framework: each recursive instance is now endowed not only with the underlying point set $P$, but also with a certain subset of edges $\Pi\subset {P\choose 2}$ which contains a large fraction of the edges spanned by the points of $P\cap K$. Thus, our recurrence can advance not only by increasing the parameter $\eps$, but also by restricting the convex sets to ``include" $\Theta\left(\eps^2n^2\right)$ edges of the progressively sparser subset $\Pi$.

 \paragraph{Definition.}
Let $\Pi\subset {P\choose 2}$ be a subset of edges spanned by the underlying $n$-point set $P$.
Let $\sigma>0$. We say that a convex set $K$ is {\it $(\eps,\sigma)$-restricted} to the graph $(P,\Pi)$ if $P\cap K$ contains a subset $P_K$ of $\lceil \eps n\rceil$ points so that the induced subgraph $\Pi_K={P_K\choose 2}\cap \Pi$ contains at least $\sigma{\lceil\eps n\rceil\choose 2}$ edges. (In particular, each $(\eps,\sigma)$-restricted set $K$ must be $\eps$-heavy with respect to $P$.) 

Notice that the choice of the set $P_K$ may not be unique, and that $K$ may enclose additional points of $P$. To simplify the presentation, in the sequel we select a unique witness set $P_K$ for every convex set $K$ that is $(\eps,\sigma)$-restricted to $(P,\Pi)$.

\medskip
At each recursive step we construct a weak $\eps$-net $Q$ for a certain family  $\K=\K(P,\Pi,\eps,\sigma)$ of convex sets which is determined by $\eps>0$, a ground set $P\subset \reals^2$ of $n$ points, a set of edges $\Pi\subseteq {P\choose 2}$, and a threshold $0<\sigma\leq 1$. This family $\K$ consists of all the convex sets $K$ that are $(\eps,\sigma)$-restricted to $(P,\Pi)$.
In what follows, we refer to $(P,\Pi)$ (or simply to $\Pi$) as the {\it restriction graph}, and to $\sigma$ as the {\it restriction threshold} of the recursive instance. 

\medskip
The topmost instance of our recurrence involves $\Pi={P\choose 2}$ and $\sigma=1$.
Each subsequent instance $\K'=\K(P',\Pi',\eps',\sigma')$ involves a larger $\eps'$ and/or a {\it much sparser} restriction graph $(P',\Pi')$. 
Each such increase in $\eps$ or decrease in the density $|\Pi|/{n\choose 2}$ is accompanied only by a comparatively mild decrease in the restriction threshold $\sigma$ which, throughout the recurrence, is bounded from below by a certain positive constant.  

Our weak $\eps$-net construction bottoms out when either
(i) the cardinality of $P$ falls below a certain threshold $n_0(\eps)$ that is close to $1/\eps^{3/2}$, or (ii) $\eps$ surpasses a certain (suitably small) constant $0<\tilde{\eps}<1$, or (iii) the density $|\Pi|/{n\choose 2}$ of the restriction graph falls below $\eps$. In the first case, $P$ comprises the desired weak $\eps$-net.
In the second case, we can use the $O\left(\left(1/\tilde{\eps}\right)^2\right)=O(1)$ bound of Alon {\it et al.} \cite{AlonSelections}. In the third  case, the discussion in Section \ref{Subsec:ProofOutline}, which we formalize in Lemma \ref{Lemma:Sparse}, yields a simpler ``near-linear" recurrence in $1/\eps$. \footnote{As a matter of fact, we have $f(\eps,\lambda,\sigma)=o\left(1/\eps^2\right)$ once the maximum density $\lambda$ falls substantially below $1$ (given that  the restriction threshold $\sigma$ is a constant). Hence, our recurrence over $\Pi$ is used to merely amplify this gain.}

\medskip
\noindent {\bf Deriving a recurrence formula for $f_2(\eps)$.}
In the course of our analysis we stick with the following notation.  We use $f(\eps,\lambda,\sigma)$ to denote the smallest number $f$ so that for any finite point set $P$ in $\reals^2$, and any subset $\Pi\subseteq {P\choose 2}$ of density $|\Pi|/{n\choose 2}\leq \lambda$, there is a point transversal of size $f$ to $\K(P,\Pi,\eps,\sigma)$.
We set $f(\eps,\lambda,\sigma)=1$ whenever $\eps\geq 1$.
Since the underlying dimension $d=2$ is fixed, for the sake of brevity we use $f(\eps)$ to denote the quantity $f_2(\eps)=f(\eps,1,1)$, and note that the trivial bound $f(\eps,\lambda,\sigma)\leq f(\eps)$ always holds.

\medskip
In the sequel, we bound the quantity $f(\eps,\lambda,\sigma)$, for $\lambda>\eps$, by a recursive expression of the general form

\begin{equation}\label{Eq:GenRecurrence}
f(\eps,\lambda,\sigma)\leq f\left(\eps,\lambda/r_0,\sigma/2\right)+O\left(\sum_{i=1}^l h_i^{1+\gamma'} \cdot f\left(\eps\cdot h_i\right)+1/\eps^{3/2+\gamma'}\right),
\end{equation}

\noindent where $l$ is a constant that does not depend on the choice of $\gamma$ in Theorem \ref{Thm:Main}, $\gamma'$ is a small constant that satisfies $0<\gamma'<\gamma$, and the parameters $r_0$ and $h_i$, for $1\leq i\leq l$, are very small (albeit fixed) degrees of $1/\eps$ that are bounded by $(1/\eps)^{\gamma'}$. 

In particular, (\ref{Eq:GenRecurrence}) will hold for $f(\eps)=f(\eps,1,1)$.
As the recurrence in the density $\lambda$ bottoms out for $\lambda\leq \eps$, applying $J=\log_{r_0} \lceil 1/\eps\rceil$ substitution steps to the first term in (\ref{Eq:GenRecurrence}) while keeping $\eps$ fixed, and then using the bound in Lemma \ref{Lemma:Sparse} (e.g., with $r=\Theta\left(1/\eps^{\gamma'}\right)$) when $\lambda$ falls below $\eps$, will result in the following bound. 

\begin{equation}\label{Eq:GenRecurrenceSimple}
f(\eps)=O\left(r\cdot f\left(\eps\cdot r\right)+\sum_{i=1}^{l} h_i^{1+\gamma'} \cdot f\left(\eps\cdot h_i\right)+1/\eps^{3/2+\gamma'}\right).
\end{equation}

\noindent Notice that in each of the intermediate substitutions, the restriction threshold $\sigma$ in $f(\eps,\lambda,\sigma)$ remains bounded from below by $2^{-J}=\Theta(1)$.

As was previously mentioned, the recurrence in $0<\eps<1$ bottoms out when it bypasses a certain constant threshold $0<\tilde{\eps}<1$. With a sufficiently small (albeit, constant) choice of $\tilde{\eps}$ which too depends on $\gamma$, the standard and fairly general induction argument (as presented, e.g., in \cite{MatWag04,EnvelopesHigh} and \cite[Section 7.3.2]{SA}) shows that recurrences of the general form of (\ref{Eq:GenRecurrenceSimple}) solve to $f(\eps)=O\left(1/\eps^{\frac{3}{2}+\gamma}\right)$, where the constant of proportionality is super-exponential in $1/\gamma$. 

\medskip

\subsection{Geometric essentials: Arrangements and strong $\eps$-nets}\label{Subsec:Essentials}
\paragraph{Strong $\eps$-nets.} Let $X$ be a (finite) set of elements and $\F\subset 2^X$ be a set of hyperedges spanned by $X$. 
A {\it strong $\eps$-net} for the hypergraph $(X,\F)$
is a subset $Y\subset X$ of elements so that $F\cap Y\neq \emptyset$ is satisfied for all hyperedges $F\in \F$ with $|F|\geq \eps n$.

\paragraph{Definition.} Let $X$ be a set of $n$ elements, and $r>0$ be an integer. An {\it $r$-sample} of $X$ is a subset $Y\subset X$ of $r$ elements chosen at random from $X$, so that each such subset $Y\in {X\choose r}$ is selected with uniform probability $1/{n\choose r}$.

\medskip
The Epsilon-Net Theorem of Haussler and Welzl \cite{HW87} states that any such hypergraph $(X,\F)$, that is drawn from a so called range space of a bounded VC-dimension $D>0$, admits a strong $\eps$-net $Y$ of cardinality $r=O\left(\frac{D}{\eps}\log \frac{D}{\eps}\right)$. 
Moreover, such a net $Y$ can obtained, with probability at least $1/2$, by choosing an $r$-sample of $X$.

In particular, this implies the following result.

\begin{theorem} \label{Thm:StrongNet}
Let $P$ be a finite set of points in $\reals^2$, 
then one can pierce all the $\eps$-heavy triangles with respect to $P$ using a net $Q^\triangle(P,\eps)$ of cardinality $O\left(\frac{1}{\eps}\log \frac{1}{\eps}\right)$.
\end{theorem}

\paragraph{Cuttings of families of lines.} An important corollary of the Epsilon-Net Theorem is the existence of {\it $\vartheta$-cuttings} of finite families $\L$ of lines in $\reals^2$ -- decompositions of $\reals^2$ into $O^*(\vartheta^{-2})$ interior-disjoint (and possibly unbounded) convex polygonal regions, which are called {\it cells}, so that any cell is delimited by $O(1)$ edges and its interior is crossed by at most $\vartheta |\L|$ lines. A two-stage such construction of worst-case size $O\left(\vartheta^{-2}\right)$ was obtained by Chazelle and Friedman \cite{Cuttings}; it uses the so called Exponential Decay Lemma to control the number of cells that arise in the secondary subdivision.

In what follows, we lay out a simpler cutting which is based on trapezoidal subdivisions of random line arrangements, and loosely corresponds to the first stage in the optimal cutting of Chazelle and Friedman.

\paragraph{Arrangements of lines in $\reals^2$.} 
Our divide-and-conquer approach uses cells in the arrangement of lines that are sampled at random from among the lines spanned by the edges of our restriction graph $(P,\Pi)$. 

To simplify the exposition, we can assume that the points of $P$ are in a general position.
In particular, no three of them are collinear, and no two of them span a vertical line.\footnote{To construct a weak $\eps$-net for a degenerate point set $P$, we perform a routine symbolic perturbation of $P$ into a general position. A weak $\eps$-net with respect to the perturbed set would immediately yield such a net with respect to the original set.}

\paragraph{Definition.} Any finite family $\L$ of $m$ lines in $\reals^2$ induces the {\it arrangement} $\A(\L)$ -- the partition of $\reals^2$ into 2-dimensional {\it cells}, or {\it $2$-faces} -- maximal connected regions of $\reals^2\setminus \left(\bigcup \L\right)$. Each of these cells is a convex polygon whose boundary is composed of {\it edges} -- portions of the lines of $\L$, which connect {\it vertices} -- crossings among the lines of $\L$.
The {\it complexity} of a cell is the total number of edges and vertices that lie on its boundary.

\begin{lemma}\label{Lemma:Triangle}
Let $L_1,L_2$ and $L_3$ be three lines in $\reals^2$, and $\Delta\subset \reals^2\setminus (L_1\cup L_2\cup L_3)$ be a cell in their arrangement. For each $1\leq i\leq 3$, let $L_i^-$ and $L_i^+$ be the two halfplanes of $\reals^2\setminus L_i$ so that $\Delta\subset L_i^-$ (see Figure \ref{Fig:ThreeLines} (left)). Suppose that each line $L_i$ contains a bounary edge $e_i$ of $\Delta$ so that the three edges $e_1,e_2$, and $e_3$, appear in this clockwise order along the boundary of $\Delta$. Then for any convex set $K$ that meets all the three sides $e_1,e_2,e_3$ of $\Delta$, and any point $p\in K\cap L_1^+\cap L_2^-\cap L_3^-$, the segment between $L_2\cap L_3$ and $p$ must cross $L_1$ within the interval $K\cap L_1$.\footnote{In the sequel, we apply the lemma only in the special case where $L_1$ and $L_3$ are vertical lines.}
\end{lemma}


\begin{figure}
    \begin{center}
        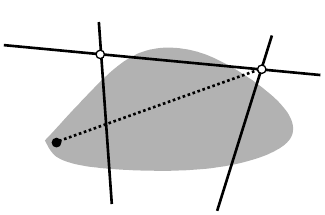\hspace{2cm}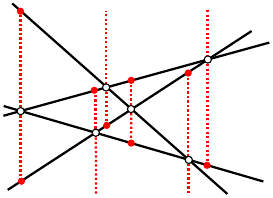
        \caption{\small Left: Lemma \ref{Lemma:Triangle} -- the set $K$ meets the three boundary edges $e_1,e_2,e_3$ of the cell $\Delta=L_1^-\cap L_2^-\cap L_3^-$. The point $p\in K$ lies in $L_1^+\cap L_2^-\cap L_3^-$. The segment between $p$ and $L_2\cap L_3$ crosses $K\cap L$. Right: The trapezoidal decomposition $\Sigma(\L)$.}
        \label{Fig:ThreeLines}
    \end{center}
\end{figure}

\paragraph{The trapezoidal decomposition.}  We further subdivide each cell $\Delta$ of the above arrangement $\A(\L)$ by raising a vertical wall from every boundary vertex of $\Delta$ that is not $x$-extremal (i.e., if the vertical line through the vertex enters the interior of $\Delta$); see Figure \ref{Fig:ThreeLines} (right). As is easy to check, the resulting decomposition $\Sigma(\L)$ is composed of $O\left(m^2\right)$ open trapezoidal cells. The boundary of each cell $\mu$ in $\Sigma(\L)$ consists of at most 4 edges\footnote{Some of the trapezoidal cells can be triangles, or unbounded.}, including at most 2 vertical edges, and the at most 2 other edges that are contained in non-vertical lines of $\L$.

\begin{theorem}\label{Theorem:SampleLines}
Let $\L$ be a family of $m$ lines in $\reals^2$, and $0<r\leq m$ integer. Then, with probability at least $1/2$, an $r$-sample $\R\in {\L\choose r}$ of $\L$ crosses every segment in $\reals^2$ that is intersected by at least $C(m/r)\log r$ lines of $\L$. 
Here $C>0$ is a sufficiently large constant that does not depend on $m$ or $r$.
\end{theorem}

The proof of Theorem \ref{Theorem:SampleLines} can be found, e.g., in \cite{ChazelleBook}. It is established by applying the Epsilon Net Theorem to the range space in which every vertex set is a finite family $\L$ of lines in $\reals^2$, and each hyperedge consists of all the lines in $\L$ that are crossed by some segment in $\reals^2$.

An easy consequence of Theorem \ref{Theorem:SampleLines} is that, with probability at least $1/2$, every (open) trapezoidal cell of the induced vertical decomposition $\Sigma(\R)$ is crossed by at most $4C(m/r)\log r$ lines of $\L$. In other words, it serves as a $\left(\frac{4C\log r}{r}\right)$-cutting of $\L$.
Despite a marginally sub-optimal bound on the number of cells (in the terms of $\vartheta:=4C\log r/r$), the simplicity of $\Sigma(\R)$ will prove beneficial for our ad-hoc argument in Section \ref{Subsec:2}.

\paragraph{The zone.} Let $\Sigma$ be a family of open cells in $\reals^2$ (e.g., the above arrangement $\A(\L)$ or its refinement $\Sigma(\L)$). 
The {\it zone} of a convex set $K\subset \reals^2$ in $\Sigma$  is the subset of all the cells in $\Sigma$ that intersect $K$. 

The {\it crossing number} of a convex set $K$ with respect to $\Sigma$ is the cardinality of its zone within $\Sigma$, that is, the number of the cells in $\Sigma$ that are intersected by $K$.

\paragraph{Definition.} For every pair $p,q\in \reals^2$ let $L_{p,q}$ denote the line through $p$ and $q$. We say that the line $L_{p,q}$ is {\it spanned} by the segment $pq$, and that $pq$ is {\it supported} by $L_{p,q}$. 
Given an $n$-point set $P$ with an edge set $\Pi\subset {P\choose 2}$, let 
$$
\L(\Pi):=\{L_{p,q}\mid \{p,q\}\in \Pi\}
$$
 
\noindent be the set of all the lines spanned by the edges of $\Pi$.
If the underlying restriction graph $(P,\Pi)$ is clear from the context, we resort to a simpler notation $\L:=\L(\Pi)$.

\paragraph{Decomposing $\reals^2$ into vertical slabs.}
For any $n$-point set $P$, and any integer $r>0$, we fix a collection $\Y(P,r)$ of $r$ vertical lines so that every vertical slab of the arrangement $\A(\Y(P,r))$ contains between $\lfloor n/(r+1)\rfloor$ to $\lceil n/(r+1)\rceil$ points of $P$, and no line of $\Y(P,r)$ passes through a point of $P$. (The two extremal slabs of $\A(\Y(P,r))$ are halfplanes, and each of them is delimited by a single line of $\Y(P,r)$.) 

In what follows, we use $\Lambda(P,r)$ to denote the above slab decomposition $\A(\Y(P,r))$.

We say that a segment $pq\subset \reals^2$ crosses a slab $\tau\in \Lambda(P,r)$ {\it transversally} if $pq$ intersects the interior of $\tau$, and none of its endpoints $p,q$ lies in $\Delta$; see Figure \ref{Fig:Slabs}.

\begin{figure}[htb]
    \begin{center}
      
        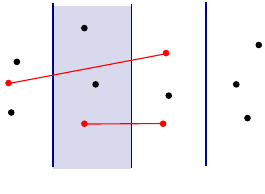
        \caption{\small The vertical lines of $\Y(P,r)$ determine a decomposition $\Lambda(P,r)$ of $\reals^2$ into $r+1$ vertical slabs. (In the depicted scenario, we have $r=3$.) The segment $pq$ crosses the slab $\tau\in \Lambda(P,r)$ transversally, while $p'q'$ does not.}
        \label{Fig:Slabs}
    \end{center}
\end{figure}

We say that a convex set $K$ is {\it $\eps'$-crowded} in $\Lambda(P,r)$ if there a slab in $\Lambda(P,r)$ that contains at least $\eps'n$ points of $P\cap K$; otherwise, we say that $K$ is {\it $\eps'$-spread} in $\Lambda(P,r)$.\footnote{We emphasize that the $\eps'$-crowdedness of a convex set $K$ depends not only on the slabs of $\Lambda(P,r)$ but also its underlying point set $P$.}

The following main property of the decompositions $\Lambda(P,r)$ is used throughout our proof of Theorem \ref{Thm:Main}.

\begin{lemma}\label{Lemma:SampleCell}
Let $P$ be an underlying set of $n$ points in $\reals^2$ and $r>0$ be an integer. 
For each $\eps'\geq 0$ there is a set $Q(P,r,\eps')$ of
$
O\left(r\cdot f\left(\eps' \cdot r\right)\right)
$ 
points that pierce every convex set $K$ that is $\eps'$-crowded in $\Lambda(P,r)$.\footnote{To simplify the presentation, we routinely omit the constant factors within the recursive terms of the form $f(\eps\cdot hr)$ as long as these constants are much larger than $1/h$.  A suitably small choice of the constant $\tilde{\eps}>0$ (and, thereby, $\eps<\tilde{\eps}$) guarantees that $\eps$ indeed increases with each invocation of the recurrence.} 

\end{lemma}

Notice that the recursive term in Lemma \ref{Lemma:SampleCell} is essentially linear in $\eps$ for $\eps'$ close enough to $\eps$; see, e.g., \cite[Section 3]{Chazelle} for a similar recurrence.


\begin{proof}[Proof of Lemma \ref{Lemma:SampleCell}]
Assume with no loss of generality that $r<2n$, for otherwise our net consists of $P$.
Recall that each slab $\tau\in \Lambda(P,r)$ cuts out a subset $P_\tau:=P\cap \tau$ of cardinality $n_\tau:=|P_\tau|\leq \lceil n/(r+1)\rceil =\Theta(n/r)$.

The crucial observation is that each $\eps'$-crowded convex set $K$ must belong to the family $\K(P_\tau,\eps'n/n_\tau)$ for some slab $\tau$ in $\Lambda(P,r)$. 
(In particular, we can further assume that $\eps'=O(1/r)$.)
For each slab $\tau\in \Lambda(P,r)$ we recursively construct the net $Q_\tau$ for the above instance $\K(P_\tau,\eps' n/n_\tau)$.
Using the definition of the function $f(\cdot)$, and that $n_\tau=\Theta(n/r)$, it is easy to check that the total cardinality of the union $Q(P,r,\eps'):=\bigcup_{\tau\in \Lambda(P,r)}Q_\tau$ is indeed $O(r\cdot f(r\cdot \eps'))$.
\end{proof}

The following lemma implies that the recursive instance $\K=(P,\Pi,\eps,\sigma)$ admits a net of size $o(1/\eps^2)$ given that the underlying restriction graph $(P,\Pi)$ is not dense (and that the restriction threshold $\sigma$ is sufficiently close $1$).

\begin{figure}[htb]
    \begin{center}
        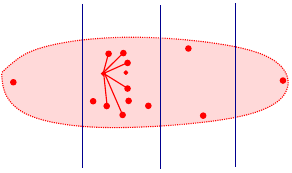\hspace{2cm}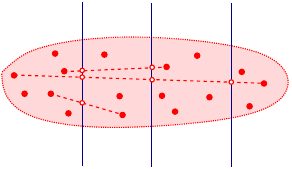
        \caption{\small Proof of Lemma \ref{Lemma:Sparse}. Left: Most edges of $\Pi_K$ do not cross any line of $\Y(P,r)$. Hence, there is a slab that contains $\Omega(\sigma \eps n)$ points of $P_K$. Right: At least half of the edges of $\Pi_K$ cross one or more lines of $\Y(P,r)$, so $K$ must be pierced by one of the nets $Q_L$.}
        \label{Fig:LemmaSparse}
    \end{center}
\end{figure}

\begin{lemma}\label{Lemma:Sparse}
Let $r\geq 1$ be an integer. Then any family $\K\subset \K(P,\Pi,\eps,\sigma)$  admits a point transversal of size

$$
O\left(r\cdot f\left(\eps\cdot \sigma\cdot r\right)+\frac{r^2|\Pi|} {\sigma \eps^2 n^2}\right).
$$
\end{lemma}

\begin{proof}
Assume with no loss of generality that $|P|\geq 2r$, for otherwise the claim follows trivially.
We consider the slab decomposition $\Lambda(P,r)$ and apply Lemma \ref{Lemma:SampleCell} with $\eps'=\sigma \eps/4$ to obtain a net $Q(P,r,\eps')$ of size 
$O\left(r\cdot f\left(\eps\cdot \sigma\cdot r\right)\right)$
that pierces every set $K\in \K$ that is $\eps'$-crowded in $\Lambda(P,r)$. 

In addition, for each vertical line $L\in \Y(P,r)$ we construct an auxiliary net $Q_L$ by choosing every $\lceil\sigma{\lceil \eps n\rceil \choose 2}/(2r)\rceil$-th crossing point of $L$ with the edges of $\Pi$. Notice that 

$$
\sum_{L\in \Y(P,r)}|Q_L|=O\left(\frac{r^2|\Pi|}{\sigma\eps^2n^2}\right)
$$

It suffices to check that every convex set $K\in \K$ is stabbed by at least one of the above nets. To this end, we distinguish between two cases.
\begin{enumerate}

\item If at least half of the segments of $\Pi_K={P_K\choose 2}\cap \Pi$ do not cross any line of $\Y(P,r)$, we find a point $p\in P_K$ so that at least $2\sigma{\lceil\eps n\rceil\choose 2}/\lceil\eps n\rceil\geq \sigma\eps n/4$ of its neighbors in the graph $(P_K,\Pi_K)$ lie in the same slab $\tau\in \Lambda(P,r)$ that contains $p$. Hence, $K$ is $\eps'$-crowded in $\Lambda(P,r)$ and, therefore, pierced by a point of $Q'$. See Figure \ref{Fig:LemmaSparse} (left).

\item At least half of the segments of $\Pi_K$ cross a line of $\Y(P,r)$. Since there are at least $(\sigma/2){\lceil\eps n\rceil\choose 2}$ intersection points between the edges of $\Pi_K$ and the lines of $\Y(P,r)$, there must be a line $L\in \Y(P,r)$ which contains at least $\sigma{\lceil\eps n\rceil\choose 2}/(2r)$ of these intersections. Hence, $K$ is hit by the corresponding net $Q_L$. See Figure \ref{Fig:LemmaSparse} (right). \end{enumerate}
\end{proof}

As mentioned in Section \ref{Subsec:RecursiveFramewk}, throughout our analysis $\sigma$ remains bounded from below by a certain positive constant, and we apply Lemma \ref{Lemma:Sparse} with $r$ that is a very small (albeit, fixed) constant power $1/\eps$. (In particular, $r$ is much larger than $1/\sigma$.)
Notice that this yields the following bound 
\begin{equation}\label{Eq:RecurrenceDensity}
f(\eps,\lambda,\sigma)=O\left(r\cdot f(\eps\cdot \sigma\cdot r)+\frac{r^2\lambda}{\sigma \eps^2}\right)
\end{equation}
in which the recursive term on the right side is essentially linear in $1/\eps$, and the constants of proportionality that are hidden by the $O(\cdot)$-notation do not depend on $\eps,\sigma$, and $\lambda$. Moreover, the non-recursive term is $o(1/\eps^2)$ provided that the density $\lambda$ is substantially smaller than $1$, and it is close to $1/\eps$ if $\lambda\leq \eps$.
A standard inductive approach to solving recurrences of this kind is presented, e.g., in \cite{Envelopes3D} and \cite[Section 7.3.2]{SA}.


\section{Proof of Theorem \ref{Thm:Main}}\label{Sec:Main}

To establish Theorem \ref{Thm:Main}, we fix $\gamma>0$ and show that $f(\eps)=O\left(1/\eps^{3/2+\gamma}\right)$. To this end, we first derive a recurrence formula of the general form (\ref{Eq:GenRecurrence}) for the quantity $f(\eps,\lambda,\sigma)$, where $\lambda>\eps$. As mentioned in Section \ref{Subsec:RecursiveFramewk}, the final recurrence for $f(\eps)$, of the form (\ref{Eq:GenRecurrenceSimple}), will follow by iterating this recurrence for $f(\eps,\lambda,\sigma)$ with $\sigma=\Theta(1)$, and then plugging in the bound of Lemma \ref{Lemma:Sparse} for $\lambda\leq \eps$. \footnote{Though we have $\sigma=\Theta(1)$ in all the subsequent applications of our recurrence for $f(\eps,\lambda,\sigma)$, starting with $f(\eps)=f(\eps,1,1)$, the dependence on the restriction threshold $\sigma$ will be spelled out throughout our analysis.}


To obtain the desired recurrence for $f(\eps,\lambda,\sigma)$ for $\lambda>\eps$, we bound the piercing number of the family $\K:=\K(P,\Pi,\eps,\sigma)$ for an arbitrary choice of the finite point set $P\subset \reals^2$ in general position, the parameters $0\leq \eps,\sigma\leq 1$, and the edge set $\Pi\subseteq {P\choose 2}$ that satisfies 

\begin{equation}\label{Eq:ManyEdges}
|\Pi|/{|P|\choose 2}\leq \lambda. 
\end{equation}

In the course of our construction, we introduce three auxiliary parameters $r_0,s_0$, and $r_1$. The first two parameters are set to be very small degrees of $1/\eps$ that depend on $\gamma$.
To this end, we set $\eta:=\gamma/100$ and $r_0:=\left\lceil 1/\eps^{\eta^2}\right\rceil$, and choose $s_0=\Theta\left(1/\eps^\eta\right)$ with the property that $s_0+1$ is the smallest multiple of $r_0+1$ that is larger than $r_0^{1/\eta}$. In addition, we will set $r_1:=\left\lceil \sqrt{1/\eps}\right\rceil$.

\smallskip
In what follows, we can assume that $\eps$ is bounded from above by a sufficiently small absolute constant $\tilde{\eps}>0$ which, in particular,  guarantees that the following inequality holds\footnote{In the sequel $\log x$ denotes the binary logarithm $\log_2 x$.}
\begin{equation}\label{Eq:SmallConstantEps}
10^5\log \frac{1}{{\eps}}<\frac{1}{{\eps}^\eta}.
\end{equation}

\noindent Otherwise, if $\eps\geq \tilde{\eps}$, the previous $O\left(1/{\tilde{\eps}}^2\right)=O(1)$ bound applies \cite{AlonSelections}.
We can also assume that the number of points $n=|P|$ is larger than a certain threshold $n_0(\eps)$, namely,

\begin{equation}\label{Eq:ManyPoints} 
n>n_0(\eps):=\left\lceil \frac{3200r_0r_1\log r_1}{\sigma\eps}\right\rceil.
\end{equation}
\noindent Otherwise, if $|P|\leq n_0(\eps)$, our transversal consists of $|P|\leq n_0(\eps)=o\left(1/\eps^{3/2+\gamma}\right)$ points. 

\medskip
For $\eps<\min\{\lambda,\tilde{\eps}\}$, which satisfy (\ref{Eq:SmallConstantEps}), and the sets $P$ and $\Pi\subseteq {P\choose 2}$ that satisfy (\ref{Eq:ManyEdges}) and $|P|>n_0(\eps)$, the piercing number of $\K$ will be bounded in the terms of the quantities $f(\eps,\lambda/r_0,\sigma/2)$ and $f(\eps')$, for $\eps'>\eps$. 
To this end, we gradually construct a net $Q$ which pierces every $\eps$-heavy set $K\in \K$. Our construction begins with an empty net $Q=\emptyset$ and proceeds through several stages.
At each stage we add a small number of points to the net $Q$ and immediately eliminate the already pierced convex sets from the family $\K$.
The surviving sets $K\in \K$, which have yet not been pierced by $Q$, satisfy additional restrictions which facilitate their treatment at the subsequent stages.

Our main decomposition $\Sigma=\Sigma(r_1)$ of $\reals^2$ in Section \ref{Subsec:1} is based on cells in the arrangement of an $r_1$-sample $\R_1$ of $\L=\L(\Pi)$, for a fairly large value  $r_1=\left\lceil \sqrt{1/\eps}\right\rceil$.
Informally, the lines of $\R_1$ are sampled from $\L$ so as to control the crossing number (i.e., size of the respective zone in $\Sigma(r_1)$) of an average edge $pq$ of $\Pi$. This bound readily extends to the narrow convex sets $K$ whose zones are traced by such edges $pq$.
Recall that our main argument (which was sketched in Section \ref{Subsec:ProofOutline}) requires that
the points $P_K$ of each set $K\in \K$ are in a ``sufficiently convex" position, and are substantially spread within the zone of $K$ in $\A(\R_1)$.
To this end, we employ an auxiliary slab decomposition $\Lambda(P,r_0)$ of Lemma \ref{Lemma:SampleCell} in combination with The Epsilon Net Theorem \ref{Thm:StrongNet}.

\paragraph{The roadmap.} The rest of this section is organized as follows.

\medskip
In Section \ref{Subsec:0} we construct an auxiliary slab decomposition $\Lambda(P,r_0)$, and use Lemma \ref{Lemma:SampleCell} to guarantee
that the points of our convex sets $K$ are sufficiently spread among the slabs of $\Lambda(P,r_0)$. 
This is achieved at expense of adding to $Q$ a small-size auxiliary net $Q_0$ which is provided by Lemma \ref{Lemma:SampleCell}.

In Section \ref{Subsec:1} we use the larger sample $\R_1$ of $r_1$ lines from $\L=\L(\Pi)$ to define the finer main decomposition $\Sigma(r_1)$ of $\reals^2$. 
As mentioned in Section \ref{Subsec:ProofOutline}, $\Sigma(r_1)$ is obtained by vertically subdividing the cells of $\A(\R_1)$ into trapezoidal sub-cells.
By the properties of $\Sigma(r_1)$ as a $\Theta(\log r_1/r_1)$-cutting for $\L$ \cite{Cuttings}, an average line of $\L$ crosses only $O(r_1\log r_1)$ cells of $\Sigma(r_1)$. We further ``normalize" $\Pi$ by omitting a relatively small fraction of its edges whose supporting lines in $\L$ cross too many of the cells of $\Sigma(r_1)$. We then remove from $\K$ every convex set that is not $(\eps,\sigma/2)$-restricted to the surviving graph $(P,\Pi)$.
To that end, we add to $Q$ another auxiliary net $Q_1$ which is obtained by solving a simpler recursive instance $\K(P,\Pi',\eps,\sigma/2)$, with a much sparser restriction graph $\Pi'$.

In Section \ref{Subsec:2} make sure that every remaining set $K\in \K$ is narrow in the decomposition $\Sigma(r_1)$ (in the sense described in Section \ref{Subsec:ProofOutline}) and, therefore, it crosses roughly $O(r_1\log r_1)$ of the decomposition cells.\footnote{More precisely, we ``clip" every set $K$ to a carefully chosen slab $\tau\in\Lambda(P,r_0)$, and apply a similar restriction to $\Sigma(r_1)$.} The leftover convex sets, that are not sufficiently narrow  in $\Sigma(r_1)$, are pierced by an auxiliary net $Q_2$ whose size is close to $r_1/\eps$.

In Section \ref{Subsec:3} we use the properties of $\Sigma(r_1)$ to construct the final net $Q_3$ which pierces all the remaining sets $K\in \K$ (missed by the auxiliary nets $Q_i$ of the previous stages $0\leq i\leq 2$). This is achieved through a skillful combination of the two paradigms sketched in Section \ref{Subsec:ProofOutline}. Thus, the eventual net $Q$ for our family $\K$ is given by the union $\bigcup_{i=0}^3Q_i$. 

In Section \ref{Subsec:WrapUp} combine the 
bounds of the preceding Sections \ref{Subsec:0} -- \ref{Subsec:3} to bound the cardinality of the complete net $Q$, and then derive the final recurrences for the quantities $f(\eps,\lambda,\sigma)$ and $f(\eps)$. 

\subsection{Stage 0: The slab decomposition $\Lambda(P,r_0)$} \label{Subsec:0}

At this stage we construct an auxiliary, almost constant-size slab decomposition $\Lambda(P,r_0)$, and use Lemma \ref{Lemma:SampleCell} to guarantee for each convex set $K\in \K$ that the points of $P_K$ are sufficiently spread among the slabs of $\Lambda(P,r_0)$.  This is achieved at the expense of adding to $Q$ a certain auxiliary net $Q_0$, and immediately removing from $\K$ all the sets already pierced by $Q_0$. 

To this end, we select a set $\Y(P,r_0)$ of vertical lines, as detailed in Section \ref{Subsec:Essentials}; each slab $\tau$ of the resulting arrangement $\Lambda(P,r_0)$ contains between $\lfloor n/(r_0+1)\rfloor$ and $\lceil n/(r_0+1)\rceil$ points of $P$. (As previously noted, the integer parameter $r_0=\Theta\left(1/\eps^{\eta^2}\right)$ is set to a very small degree of $1/\eps$.)  

\medskip
\noindent{\bf The net $Q_0$.} By Lemma \ref{Lemma:SampleCell}, we can pierce (and subsequently remove from $\K$) every $\left(\sigma\eps/100\right)$-crowded convex set $K$ using an auxiliary net 

\begin{equation}\label{Eq:Q_0Definition}
Q_0:=Q\left(P,r_0,\sigma\eps/100\right)
\end{equation} 

\noindent whose cardinality satisfies

\begin{equation}\label{Eq:Stage0}
|Q_0|=O\left(r_0\cdot f\left(\eps\cdot \sigma \cdot r_0\right)\right).
\end{equation}

In what follows, we can assume that every remaining set $K\in \K$ is $(\sigma\eps/100)$-spread in the slab decomposition $\Lambda(P,r_0)$.


\subsection{Stage 1: The main decomposition of $\reals^2$} \label{Subsec:1}
At this stage we construct the main decomposition $\Sigma(r_1)$ of $\reals^2$ into $O(r_1^2)$ cells. Since $\Sigma(r_1)$ is a refinement of the auxiliary slab decomposition $\Lambda(P,r_0)$, we can use the properties of $\Lambda(P,r_0)$ to show that the points of $P_K$ are sufficiently spread in the finer decomposition $\Sigma(r_1)$. In particular, a so called middle slab $\tau\in \Lambda(P,r_0)$ can be obtained for every remaining convex set $K\in \K$, which will play a quintessential role in the analysis of Sections \ref{Subsec:2} and \ref{Subsec:3}.

\smallskip
\paragraph{The decomposition $\Sigma(r_1)$.} We sample a subset $\R_1$ of $r_1=\left\lceil \sqrt{1/\eps}\right\rceil$ lines from $\L=\L(\Pi)$.
We can assume with no loss of generality that no line of $\Y(P,r_0)$ passes through a vertex of $\A(\R_1)$.\footnote{If $m<r_1$ then we obtain the desired decomposition by choosing $\R_1=\L$. Note that the lines of $\R_1$ are not necessarily in a general position: many of them can pass through the same point of $P$. Nevertheless, there exist at most $2r_1$ such points in $P$ that lie on one or more lines of $\R_1$.} 
To simplify the exposition, we add the vertical lines of $\Y(P,r_0)$ to $\R_1$, so the arrangement $\A(\R_1)$ is a refinement of $\Lambda(P,r_0)$.

We then construct the trapezoidal decomposition $\Sigma(\R_1)$ of $\A(\R_1)$ which was described in Section \ref{Subsec:Essentials}; see Figure \ref{Fig:ZoneEdge}. We further subdivide each cell $\hat{\mu}\in \Sigma(\R_1)$ (where necessary) into sub-trapezoids $\mu$ so that $|P\cap \mu|\leq n/r_1^2$; this can be achieved using $O\left(\lfloor r_1^2|P\cap \hat{\mu}|/n\rfloor\right)$ additional vertical walls. Furthermore, we can assume that none of these walls coincides with a point of $P$.

\begin{figure}[htb]
    \begin{center}
       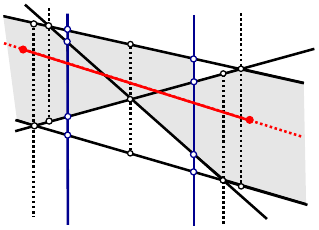
        \caption{\small Our vertical decomposition $\Sigma(r_1)$ which incorporates the lines of $\Y(P,r_0)$. The zone of a line $L_{p,q}\in \L$ is shaded. Following the removal of $\Pi_{>t}$, every remaining edge $pq\in \Pi$ crosses at most $t=r_0r_1\log r_1$ cells of $\Sigma(r_1)$.}
        \label{Fig:ZoneEdge}
    \end{center}
\end{figure}

A standard calculation (see, e.g., \cite{Chan}) shows that the resulting finer partition $\Sigma(r_1)$ encompasses a total of $O(r_1^2)$ trapezoids. Since $\Sigma(r_1)$ is a refinement of $\Sigma(\R_1)$, the relative interior of each of its cells is still crossed by $O\left((m\log r_1)/r_1\right)$ lines of $\L$, where $m$ denotes the cardinality of $\Pi$ and $\L=\L(\Pi)$. As explained in Section \ref{Subsec:Essentials}, $\Sigma(\R)$ along with its refinement $\Sigma(r_1)$ serve as simple $O\left(\frac{\log r_1}{r_1}\right)$-cuttings with respect to $\L$.

\medskip
\noindent{\bf Refining the restriction graph $\Pi$.} 
Since every trapezoidal cell of $\Sigma(r_1)$ is crossed by $O((m\log r_1)/r_1)$ lines of $\L$, the zone of an ``average" line in $\L=\L(\Pi)$ consists of $O(r_1\log r_1)$ cells of $\Sigma(r_1)$.

To eliminate the edges of $\Pi$ whose supporting lines in $\L$ deviate ``too far" from the average behaviour with respect to our decomposition $\Sigma(r_1)$, we set $t:=r_0r_1\log r_1$ and use $\L_{>t}\subset \L$ to denote the subset of all the lines in $\L$ that cross more than $t$ (open) cells of $\Sigma(r_1)$.

\begin{proposition}\label{Prop:Sparser} We have that
$$
|\L_{>t}|=O\left(\frac{m}{r_0}\right).
$$
\end{proposition}

For the sake of completeness, we spell out the fairly standard proof of Proposition \ref{Prop:Sparser}.

\begin{proof}
Since any trapezoidal cell $\mu$ in $\Sigma(r_1)$ is crossed by $O\left((m\log r_1)/r_1\right)$ lines of $\L$, the bipartite graph of pairwise intersections between the lines of $\L$ and the cells of $\Sigma(r_1)$ contains
$$
O\left(r_1^2 \cdot \frac{m \log r_1}{r_1}\right)=O(r_1 m\log r_1)
$$
edges. Since every line of $\L_{>t}$ contributes at least $t=r_0r_1\log r_1$ intersections, the claim now follows by applying the pigeonhole principle (or Markov's inequality).
\end{proof}

\medskip
\noindent{\bf The net $Q_1$.} Let $\Pi_{>t}$ be the set of edges that span the lines of $\L_{>t}$. 
Consider the recursive instance 
$$
\K_{>t}:=\K(P,\Pi_{>t},\eps,\sigma/2).
$$

Using the bound of Proposition \ref{Prop:Sparser} on $|\Pi_{>t}|=|\L_{>t}|$, we can pierce the sets of $\K_{>t}$ by an auxiliary net $Q_1$ of size\footnote{For the sake of brevity, in our asymptotic analysis we omit the multiplicative constants within the arguments $\lambda'$ of the recursive terms $f(\eps,\lambda',\sigma')$.}

\begin{equation}\label{Eq:Stage1}
f\left(\eps,\frac{|\Pi_{>t}|}{{n\choose 2}},\frac{\sigma}{2}\right)
\leq f\left(\eps,\frac{m}{r_0{n\choose 2}},\frac{\sigma}{2}\right)\leq f\left(\eps,\frac{\lambda}{r_0},\frac{\sigma}{2}\right).
\end{equation}

\noindent We immediately add the points of $Q_1$ to our net $Q$, and remove the sets of $\K_{>t}$ from our family $\K$. Note that choosing $r_0$ to be a very small (albeit, constant) positive power of $1/\eps$ guarantees that the recurrence (\ref{Eq:Stage1}) in the maximum density $\lambda\geq m/{n\choose 2}$ is invoked only a fixed number of times before $\lambda$ falls below $\eps$; thus, $\sigma$ remains bounded from below by a sufficiently small constant.

\medskip

Notice that every remaining set $K\in \K$ belongs to the family 
$\K\left(P,\Pi\setminus\Pi_{>t},\eps,\sigma/2\right)$. We thus remove the edges of $\Pi_{>t}$ from $\Pi$ and, accordingly, remove the edges of $\Pi_{>t}$ from the subsets $\Pi_K:=\Pi\cap {P_K\choose 2}$ induced by all the convex sets $K\in \K$. In doing so, we stick with the same remaining family $\K$ even if some of its sets $K\in \K$ are only $\left(\eps,\sigma/2\right)$-restricted with respect to the refined graph $(P,\Pi)$.

\medskip
\paragraph{Choosing a middle slab in $\Lambda(P,r_0)$.} Denote 
\begin{equation}\label{Eq:DefEps0}
\eps_0:=\frac{\sigma\eps}{100r_0}.
\end{equation}

\noindent {\bf Definition.} Let $K\in \K$ be a convex set. We say that a slab $\tau$ in the auxiliary decomposition $\Lambda(P,r_0)$ is a {\it middle slab} with respect to $K$ if the following conditions are satisfied (see Figure \ref{Fig:MiddleSlab}):

\bigskip
(M1) $\eps_0 n\leq |P_K\cap \tau|\leq \sigma \eps n/100$, and

\bigskip
(M2) At least $\sigma{\lceil\eps n\rceil\choose 2}/(8r_0)=\Omega\left(\sigma\eps^2n^2/r_0\right)$ of the edges of $\Pi_K$ cross $\tau$ transversally. \footnote{See Section \ref{Subsec:Essentials} for the definition. Here $\Pi_K$ denotes the induced sub-graph $\Pi\cap {P_K\choose 2}$ {\it after} removing the edges of $\Pi_{>t}$.} (In particular, there is {\it at least one} such edge $pq\in \Pi_K$ that crosses $\tau$ transversally.)

\begin{figure}[htb]
    \begin{center}
        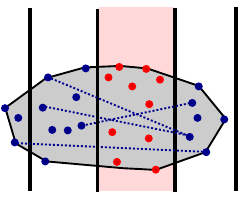
        \caption{\small The slab $\tau\in \Lambda(P,r_0)$ is a middle slab for $K$. The depicted edge $pq\in \Pi_K$ crosses $\tau$ transversally.}
        \label{Fig:MiddleSlab}
    \end{center}
\end{figure}

 Informally, the second property (M2) yields $\Omega\left(\eps^2 n^2/r_0\right)$ edges in $\Pi_K$ that can be used to trace the zone of $K$ within the restriction of $\Sigma(r_1)$ to $\tau$. In what follows, this facilitates the study of generalized incidences between such convex sets $K$ and the points of $P_\tau=P\cap \tau$.

\begin{proposition}\label{Prop:MiddleCell} For each remaining convex set $K\in \K$, there is at least one middle slab in $\Lambda(P,r_0)$.
\end{proposition}

\begin{proof} 
Fix a set $K\in\K$. 
By definition, any convex set $K\in \K$ with at least $\sigma\eps n/100$ points in a single slab $\tau\in \Lambda(P,r_0)$ is $\left(\sigma\eps/100\right)$-crowded and, therefore, already pierced by the net $Q_0=Q\left(P,r_0,\sigma\eps/100\right)$ of Section \ref{Subsec:0}. Hence, the second inequality in (M1) holds for {\it any} slab $\tau\in \Lambda(P,r_0)$.

Let $\Lambda_K$ be the set of all the slabs $\tau$ in $\Lambda(P,r_0)$ that are intersected by $K$ and satisfy $|P_K\cap \tau|\geq \eps_0n=\sigma\eps n/(100r_0)$. Notice that every slab of $\Lambda_K$ satisfies condition (M1), and the points in the slabs of $\Lambda(P,r_0)\setminus \Lambda_{K}$ are involved in a total of at most  
$$
\frac{\sigma \eps n}{100r_0}\cdot (r_0+1)\cdot \lceil \eps n\rceil \leq \frac{\sigma}{4}{\lceil\eps n\rceil\choose 2}
$$ 
adjacencies with the edges of $\Pi_K$. Using that $K$ is $(\eps,\sigma/2)$-restricted with respect to the refined graph $(P,\Pi)$, so that $|\Pi_K|=|{P_K\choose 2}\cap \Pi|\geq \frac{\sigma}{2} {\lceil\eps n\rceil\choose 2}$, we obtain a subset $\Pi'_K\subseteq \Pi_K$ of at least $\frac{\sigma}{4}{\lceil\eps n\rceil\choose 2}$ edges so that both of their endpoints lie in the slabs of $\Lambda_K$.


If no cell in $\Lambda_K$ satisfies condition (M2), we obtain at least
$|\Pi'_K|/2>\frac{\sigma}{8}{\lceil\eps n\rceil\choose 2}$ edges of $\Pi'_K$ so that none of them has a transversal crossing with a slab of $\Lambda_K$.
 Thus, by the pigeonhole principle, there must be a slab $\tau\in \Lambda_K$ and a point $p\in P_K\cap \tau$ so that at least $\sigma {\lceil\eps n\rceil\choose 2}/(4\lceil\eps n\rceil)$ of its neighbors in the graph $\Pi'_K$ lie either in $\tau$ or in one of its (at most) two neighboring slabs within $\Lambda_K$. (Notice that these slabs need not be consecutive in $\Lambda(P,r_0)$ or $\Lambda_K$.)
Since one of these three slabs of $\Lambda_K$ must then contain at least $\sigma {\lceil\eps n\rceil\choose 2}/(12\lceil\eps n\rceil)$ neighbors of $p$ in $\Pi'_K$ (and we have $n>n_0(\eps)$, as defined in (\ref{Eq:ManyPoints})), the convex set $K$ is $(\sigma\eps/100)$-crowded in $\Lambda(P,r_0)$. Hence, $K$ must have been pierced by the net $Q_0=Q\left(P,r_0,\sigma\eps/100\right)$ of Stage 0, and already removed from $\K$. This contradiction establishes the claim.
\end{proof}

To recap, for every remaining convex set $K\in \K$ (which is missed by the combination $Q_0\cup Q_1$) there is at least one middle slab $\tau\in \Lambda(P,r_0)$. 
Furthermore, each of the edges $pq\in \Pi_K$ (at least $(\sigma/2){\lceil\eps n\rceil\choose 2}$ in number) that cross $\tau$ transversally by property (M2), meets the interiors of at most $t=r_0r_1\log r_1$ cells of the decomposition $\Sigma(r_1)$.


In the following Section \ref{Subsec:2} we use these two properties to guarantee that every set $K\in \K$ intersects at most $t=r_0r_1\log r_1$ cells of $\Sigma(r_1)$ within some middle slab $\tau$ of $K$. As before, this is achieved at expense of adding an additional small-size auxiliary net to $Q$.

\subsection{Stage 2: Controlling the crossing number in $\Sigma(r_1)$}\label{Subsec:2}
\medskip
\noindent{\bf Definition.}  To simplify our exposition, for each $K\in \K$ we fix a middle slab $\tau\in \Lambda(P,r)$ with an edge $pq\in \Pi_K$ that crosses $\tau$ transversally. (By condition (M2), such an edge $pq$ exists and can be chosen in $\Omega(\sigma\eps^2n^2/r_0)$ possible ways.) In what follows, we refer to $\tau$ as the {\it principal middle slab}, and to $pq$ as the {\it principal edge}, of $K$.

\medskip
For each slab $\tau\in \Lambda(P,r_0)$ we consider the subfamily $\K_\tau\subset \K$ of all the convex sets $K\in \K$ so that $\tau$ is their principal middle slab. 
By Proposition \ref{Prop:MiddleCell}, we have 
$\K=\biguplus_{\tau\in \Lambda(P,r_0)}\K_\tau$. 

\medskip
In Section \ref{Subsec:3}, we will use the decomposition $\Sigma(r_1)$ to construct a small-size net $Q_\tau$ for each sub-family $\K_\tau$. 
To this end, for every slab $\tau\in \Lambda(P,r_0)$ we consider the restriction 
$$
\Sigma_\tau:=\{\mu\in \Sigma(r_1)\mid \mu\subset \tau\}.
$$

\begin{figure}[htb]
    \begin{center}
      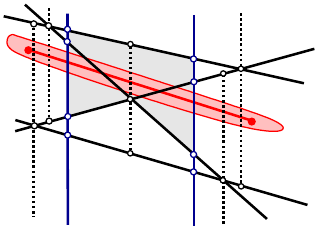
        \caption{\small The set $K$ is narrow in $\Sigma_\tau$ because $K\cap \tau$ is contained in the zone of the principal edge $pq\in \Pi_K$ which crosses $\tau$ transversally. (The cells of the zone of $K\cap \tau$ within $\Sigma_\tau$ are shaded.)}
        \label{Fig:NarrowRestrict}
    \end{center}
\end{figure}

\paragraph{Definition.} Let $K\in \K$ be a convex set, and let $\tau$ be the principal middle slab of $K$.
We say that $K$ is {\it narrow} (in $\Sigma_\tau$) if the restriction $K\cap \tau$ is contained in the zone of its principal edge $pq\in \Pi_K$ of $K$ within $\Sigma_\tau$. (By definition, the cells of this zone lie in the zone of $pq$ within the arrangement $\A(\R_1)$.) See Figure \ref{Fig:NarrowRestrict}.


\medskip
Informally, the narrowness of $K\in \K_\tau$ means that its behaviour is ``line-like" in $\Sigma_\tau$, so the zone of $K$ in $\Sigma_\tau$ can be completely ``read off" from its principal edge $pq\in \Pi_K$.


\begin{proposition}\label{Prop:Zone}
Let $\tau$ be a slab of $\Lambda(P,r_0)$, and let $K\in\K_{\tau}$ be a narrow convex set. Then $K$ intersects at most $r_0r_1\log r_1$ cells of $\Sigma_{\tau}$.
\end{proposition}
\begin{proof}
Let $pq\in \Pi_K$ be the principal edge of $K$. Since the $pq$ does not belong to the set $\Pi_{>t}$ which we removed at Stage 1, its zone in $\Sigma(r_1)$ (and, in particular, in the restriction $\Sigma_\tau$ of  $\Sigma(r_1)$ to the principal middle slab $\tau$) consists of at most $r_0r_1\log r_1$ cells.  By the narrowness of $K$, these cells form the zone of $K\cap \tau$ within $\Sigma_{\tau}$. 
\end{proof}

\medskip
\noindent{\bf The net $Q_2$.} We now get rid of the sets $K\in \K$ that are not narrow.

\begin{proposition}\label{Prop:Narrow1}
There is a set $Q_2\subset \reals^2$ of cardinality $O\left(r_0^2 r_1/\eps\right)$ that, for each slab $\tau\in \Lambda(P,r_0)$, pierces every convex set $K\in \K_\tau$ that is not narrow in $\Sigma_\tau$. 
\end{proposition}

\begin{proof}


We first add to $Q_2$ all the $O(r_1^2)$ vertices of the trapezoids of $\Sigma(r_1)$.
We then add to $Q_2$ the set $X$ of the $r_0r_1$ intersection points of the $r_0$ vertical lines of $\Y(P,r_0)$ with the lines of $\R_1$, and construct an even larger family $Y\subset \bigcup \Y(P,r_0)$ by intersecting each line of $\Y(P,r_0)$ with the edges of $P\times X$. Notice that the resulting point set has cardinality at most $O(r^2_0r_1n)$, as each line of $\Y(P,r_0)$ contains at most $r_0r_1n$ crossing points. Let $C_2>0$ be a sufficiently small constant that will be determined in the sequel. For each line $L\in \Y(P,r_0)$ we add to $Q_2$ every $\lceil C_2\eps n\rceil$-th point of $L\cap Y$, for a total of $O(r_0^2r_1/\eps)$ such points. 

Since $r_1=\Theta\left(\sqrt{1/\eps}\right)$, the overall cardinality of our auxiliary net $Q_2$ is bounded by $O\left(r_1^2+r_0^2r_1/\eps\right)=O\left(r_0^2r_1/\eps\right)$. It, therefore, suffices to check that $Q_2$ satisfies the asserted properties with a suitably small choice of the constant $C_2>0$.
To this end, we fix a slab $\tau\in \Lambda(P,r_0)$ and a convex set $K\in 
\K_\tau$ that is missed by $Q_2$.

Let $pq\in \Pi_K$ be the principal edge of $K$.
Since $K$ is missed by the points of $Q_2$, and $pq$ crosses both of the lines of $\Y(P,r_0)$ that delimit $\tau$, the edge $pq$ is not contained in a line of $\R_1$, and it cannot pass through a vertex of $\Sigma(r_1)$.
Assume with no loss of generality that $q$ lies to the right of $p$.


\begin{figure}[htb]
    \begin{center}
        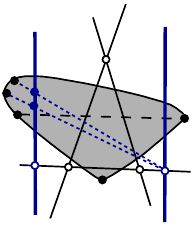\hspace{2cm}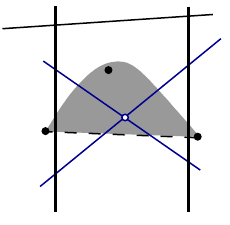\hspace{2cm}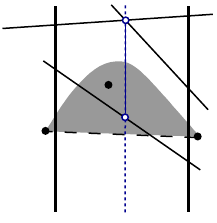
        \caption{\small Proof of Proposition \ref{Prop:Narrow1}. Showing that every set $K\in \K_{\tau}$, that is not narrow in $\Sigma_\tau$, is pierced by a point of $Q_2$.
        Left: In the first case, the supercell $\rho$ of $u$ in $\A(\R_1)$ is separated from $pq\cap \tau$ by a line $L_\mu\in \R_1$. The interval $K\cap L_{0}$ is crossed by every edge that connects a point of $P_K\cap L_0^-$ to the vertex $L_\mu\cap L_1\in X$. Center: In the second case,  $\rho$ lies in the only wedge of $\reals^2\setminus (L_\mu\cup L_{\mu'})$ that is missed by $pq\cap \tau$. The vertex $L_{\mu}\cap L_{\mu'}$ belongs to $K$. Right: The cell $\mu$ is separated from $pq\cap \rho$ by a vertical wall $g$, so $K$ contains at least one of the endpoints of $g$.}
        \label{Fig:ForceNarrow1}
    \end{center}
\end{figure}


Assume for a contradiction that there is a point $u\in K\cap \tau$ that lies in a cell $\mu\in \Sigma_\tau$ outside the zone of $pq$. Let $\rho$ be the parent cell of $\mu$ in the arrangement $\A(\R_1)$.
Let $L_0$ (resp., $L_1$) be the vertical line adjacent to $\tau$ from the left (resp., right).
Let $L^-_0$ (resp., $L^+_1$) denote the halfplane of $\reals^2\setminus L_{0}$ (resp., $\reals^2\setminus L_1$) containing $p$ (resp., $q$). 
We distinguish between three possible cases as illustrated in Figure \ref{Fig:ForceNarrow1}.

\begin{enumerate}
\item Both cells $\mu$ and $\rho$ are separated from $pq\cap \tau$ by a line $L_\mu\in \R_1$ that misses $pq\cap \tau$. 

Since $K$ is not pierced by $X$, $K\cap L_0^-$ must lie to the same side of $L_\mu$ as $p$ (or, else, $K$ would contain the point $L_0\cap L_\mu\in Q_2$), and a symmetric property must hold for $K\cap L_1^+$. Since $|P_\tau\cap K|\leq \eps n/4$ (by the property (M1) of the middle slab $\tau$ with respect to $K$), at least one of the subsets $P_K\cap L_0^-,P_K\cap L_1^+$, let it be the former set, must contain more than $\eps n/4$ points of $P_K$.
Applying Lemma \ref{Lemma:Triangle} to the cell $\Delta\subset \reals^2\setminus (L_0\cup L_1\cup L_\mu)$ that contains $pq\cap \tau$ readily implies that $L_0\cap K$ must be crossed by all the edges connecting the vertex $L_1\cap L_\mu\in X$ and the $\Omega(\eps n)$ points of $P_K\cap L_{0}^-$. Given a sufficiently small choice of $C_1$, the intercept $K\cap L$ must contain a point of $Q_2$.

\item There exist lines $L_{\mu},L'_{\mu}$, each crossing $pq\cap \tau$, so that both $\rho$ and $\mu\subset \rho$ lie in the only wedge of $\reals^2\setminus (L_\mu\cup L'_\mu)$ that does not meet $pq$. In this case, $K$ must be pierced by the vertex $L_\mu\cap L'_\mu\in X\subset Q_2$.

\item  The principal edge $pq$ crosses $\rho$ but the cell $\mu$ is separated from $pq\cap \rho$ by a vertical line $L_\mu$ which supports a vertical wall $g$ on the boundary of $\mu$. (In particular, $K$ must cross $g$.) Since $pq$ crosses $\tau$ transversally, it must cross the line $L_\mu$ which is ``sandwiched" within $\tau$, and this crossing must happen outside $g$. Therefore, and due to its convexity, $K$ must contain at least one of the endpoints of $g$, which again belong to $Q_2$ as the vertices of $\Sigma(r_1)$.
\end{enumerate}

We conclude that, in either of the above three cases, $K$ must contain a point of $Q_2$. This contradiction confirms that $K$ is indeed narrow in $\Sigma_\tau$.
\end{proof}

\noindent{\bf Remark.} A careful look at the proof of Proposition \ref{Prop:Narrow1} indicates that, for each convex set $K\in \K_\tau$ that is missed by $Q_2$, its portion $K\cap \tau$ is contained in the zone (within $\Sigma_\tau$) of {\it any} segment $pq\subset K$ that crosses $\tau$ transversally. 






\medskip
We immediately add the points of $Q_2$ to our net $Q$, and remove from $\K$ (and, thus, from each subset $\K_\tau$) every set that is pierced by $Q_2$. As a result, for every $\tau\in \Lambda(P,r_0)$, every remaining set of $\K_\tau$ is narrow in the restriction $\Sigma_\tau$ of $\Sigma(r_1)$ to the principal middle slab $\tau$ of $K$.

Combing the bound $|Q_2|=O\left(r^2_0r_1/\eps\right)$ of Proposition \ref{Prop:Narrow1} with the bounds (\ref{Eq:Stage0}) and (\ref{Eq:Stage1}) on the auxiliary nets $Q_0$ and $Q_1$ that were constructed at the previous Stages 0 and 1, so far we have added a total of
\begin{equation}\label{Eq:EarlyStages}
f\left(\eps,\frac{\lambda}{r_0},\frac{\sigma}{2}\right)+O\left(r_0\cdot f(\eps\cdot\sigma\cdot r_0)+\frac{r_0^2r_1}{\eps}\right)
\end{equation}

\noindent points to the net $Q$. As previously mentioned, choosing $r_0$ to be a very small (albeit, constant) positive power of $1/\eps$ guarantees that our recurrence (\ref{Eq:Stage1}) in $\lambda$ has only constant depth; thus, $\sigma$ remains bounded from below by a certain positive constant. Hence, the second recursive term is essentially linear in $1/\eps$. Therefore, the contribution of (\ref{Eq:EarlyStages}) to the cardinality of $Q$ is effectively dominated by the non-recursive term, which is roughly bounded by $1/\eps^{3/2}$ for  $r_0\ll r_1=\Theta(\sqrt{1/\eps})$.\footnote{For $x,y\geq 1$, the notation $x\ll y$ means that $x=O\left(y^{O(\eta)}\right)$. (For $0<x,y\leq 1$, the notation $x\ll y$ means that $1/y\ll 1/x$.)}

\medskip
\noindent{\bf Discussion.} 
Note that a more economical construction of the sets $Y_L$, for $L\in \Y(P,r_0)$, would have resulted in an auxiliary net of size $O(r_0r_1/\eps)$, and with exactly same properties as argued in Proposition \ref{Prop:Narrow1}. However, the actual polynomial dependence on $r_0$ is immaterial for the eventual recurrence that we derive for $f(\eps)$ in Section \ref{Subsec:WrapUp}.

\subsection{Stage 3: The set $P_K$ -- from the low crossing number to expansion in $\Sigma(r_1)$}\label{Subsec:3}
At this stage we complete the construction of the net $Q$ for $\K(P,\Pi,\eps,\sigma)$. 
As each convex set $K\in \K$ is equipped with the principal middle slab $\tau$ and, therefore, assigned to the respective subfamily $\K_\tau\subseteq \K$, it suffices to construct a ``local" net $Q_\tau$ for each subfamily $\K_\tau$. 
To this end, we implement the paradigm of Section \ref{Subsec:ProofOutline} within the restriction $\Sigma_\tau$ of $\Sigma(r_1)$ to $\tau$, each of whose trapezoidal cells contains at most $n/r_1^2$ points of $P_\tau=P\cap \tau$. 

\paragraph{Definition.} 
We fix a sufficiently small constant $0<\hat{C}\leq 1/120$ that will be determined in the sequel, and denote

\medskip
\begin{center}
$\displaystyle \eps_1:=\frac{\eps_0}{40\log 1/\eps}$ and $\displaystyle\hat{\eps}:=\frac{\eps_0}{8r_0 r_1 \log r_1}$
\end{center}

\bigskip
\noindent{\bf The auxiliary nets $Q(P,s_0,\eps_1/4)$ and $Q^\triangle\left(P,\hat{C}\hat{\eps}\right)$.} We first guarantee that the points of $P_K$ are in a sufficiently convex position, and that they are sufficiently spread within $\Sigma_\tau$. (The latter property is essential for guessing the splitting line $L$, whose intercept $K\cap L$ is crossed by many edges of ${P_\tau\choose 2}$.) 
To this end, we introduce two auxiliary nets.

\begin{enumerate}
\item We construct a finer slab decomposition $\Lambda(P,s_0)$ where $s_0=\Theta\left(r_0^{1/\eta}\right)$ is yet another small constant power of $1/\eps$ that was mentioned in the beginning of Section \ref{Sec:Main}. Since $s_0+1$ is a multiple of $r_0+1$, we can assume with no loss of generality that $\Lambda(P,s_0)$ is a refinement of $\Lambda(P,r_0)$, that is, we have $\Y(P,s_0)\supset \Y(P,r_0)$. Furthermore, since $s_0\ll r_1=\Theta\left(\sqrt{1/\eps}\right)$, we can add the lines of $\Y(P,s_0)$ to the sample $\R_1$ with no affect on the asymptotic properties of $\A(\R_1)$ and its vertical decomposition $\Sigma(r_1)$. 

\medskip
We then apply Lemma \ref{Lemma:SampleCell} to construct an auxiliary net $Q(P,s_0,\eps_1/4)$ that pierces every convex set that is $(\eps_1/4)$-crowded in $\Lambda(P,s_0)$. Notice that

\begin{equation}\label{Eq:SpreadS}
|Q(P,s_0,\eps_1/4)|=O\left(s_0\cdot f\left({\eps_1} \cdot s_0\right)\right)=O\left(s_0\cdot f\left(\eps\cdot\frac{s_0\cdot \sigma}{r_0\log 1/\eps}\right)\right),
\end{equation}

where the last estimate uses the definition (\ref{Eq:DefEps0}) of $\eps_0$ in Section \ref{Subsec:0}.

Upon adding $Q(P,s_0,\eps_1/4)$ to $Q$, we can assume that each remaining convex set $K$ is $(\eps_1/4)$-spread in $\Lambda(P,s_0)$.

\item We invoke Theorem \ref{Thm:StrongNet} to construct a strong $\left(\hat{C}\hat{\eps}\right)$-net  $Q^\triangle\left(P,\hat{C}\hat{\eps}\right)$ over the set $P$ with respect to triangles, and add its points to the nets $Q$ and $Q_{\tau}$. 

Notice that this step increases the cardinality of $Q$ by 
\begin{equation}\label{Eq:TriangleNet}
\left|Q^\triangle\left(P,\hat{C}\hat{\eps}\right)\right|=O\left(\frac{1}{\hat{\eps}}\log \frac{1}{\hat{\eps}}\right)=O\left(\frac{r^2_0r_1}{\eps  \sigma}\log^2 \frac{1}{\eps}\right).
\end{equation}

\noindent

Accordingly, we remove from $\K$ and $\K_{\tau}$ every convex set that contains a triangle whose interior encloses at least $\hat{C}\hat{\eps} n$
points of $P$. 
\end{enumerate}

We now establish the key properties of the remaining sets $K\in \K$, which are missed by the combination $Q(P,s_0,\eps_1/4)\cup Q^\triangle\left(P,\hat{C}\hat{\eps}\right)$. For each $K\in \K$, we study the distribution of the points of $P_K\subseteq P\cap K$ in the decomposition $\Sigma_\tau\subset \Sigma(r_1)$ of the respective principal middle slab $\tau$ of $K$. 

\bigskip
\noindent {\bf The setup.} Fix $K\in \K_\tau$. Since $\tau$ is a middle slab for $K$, it satisfies the criteria (M1) and (M2) detailed in Section \ref{Subsec:1}. Namely, we have $|P_K\cap \tau|\geq \eps_0n$ by condition (M1) and the graph $\Pi_K$ contains $\Omega\left(\frac{\sigma}{r_0}{\lceil\eps n\rceil\choose 2}\right)$ edges that cross $\tau$ transversally by condition (M2); these include the unique principal edge $pq\in \Pi_K$ of $K$. 
We also assume that $K$ is narrow in $\Sigma_\tau$. Therefore, by Proposition \ref{Prop:Zone}, the zone of $K$ in $\Sigma_{\tau}$ is comprised of the at most $r_0r_1\log r_1$ trapezoidal cells that are crossed by $pq$.

\begin{figure}
    \begin{center}
        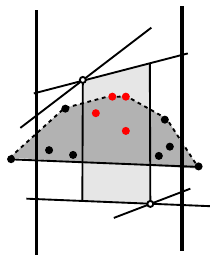
        \caption{\small The principal middle slab $\tau$ of $K$, the principal edge $pq\in \Pi_K$, the set $K^+=\conv(P^+_K\cup \{p,q\})$, and a cell $\mu\in \Sigma_\tau$. (The points of $P_K(\mu)=P_K^+\cap \mu$ are colored red.)}
        \label{Fig:KLeft}
    \end{center}
\end{figure}


\medskip
In what follows, we can assume for each $K\in \K_\tau$ that at least $\eps_0 n/2-2$ of the points of $P_K\cap \tau$ lie above the line $L_{p,q}$ from $p$ to $q$. Otherwise, $K$ is treated in a fully symmetric manner by reversing the direction of the $y$-axis. In addition, we can assume that $p$ lies to the left of $q$. (Recall that $pq$ is not contained in a line of $\R_1$, and it cannot pass through a vertex of $\Sigma_\tau$.)

\medskip
Let $P_K^+$ denote the portion of $P_K\cap \tau$ above the line $L_{p,q}$. Denote $K^+:=\conv(P_K^+\cup \{p,q\})$. Notice that  $K^+$ is supported by the line $L_{p,q}$ at its boundary edge $pq$; see Figure \ref{Fig:KLeft}. Note also that $K^+$ too is narrow in $\Sigma_\tau$.

\bigskip
\noindent{\bf Definition.} For each (open) cell $\mu\in \Sigma_{\tau}$ we denote $P_K(\mu):=P_K^+\cap \mu$ and $k_\mu:=|P_K(\mu)|$. 
We say that a cell $\mu\in \Sigma_{\tau}$ is {\it full} with respect to $K^+$ if $k_\mu\geq \hat{\eps} n\geq 100$, where the second inequality always holds due to the lower bound $n\geq n_0(\eps)$ in (\ref{Eq:ManyPoints}). Let $\Sigma_K$ denote the sub-collection of all the full cells in $\Sigma_\tau$.

\begin{proposition}\label{Eq:ManyFullCells}
At least $\eps_0n/5$ points of $P_K^+$ lie in the cells of $\Sigma_K$.
\end{proposition}
\begin{proof}
Since $K$ intersects at most $r_0r_1\log r_1$ cells of $\Sigma_\tau$, the non-full cells of $\Sigma_\tau$ contain a total of at most $\eps_0 n/8$ points of $P_K^+$.
Since at most $2r_1$ points of $P$ are contained in the lines of $\R_1$, at least $\eps_0n/2-2-\eps_0 n/8-2r_1$ points of $P_K^+$ lie in the cells of $\Sigma_K$.
The claim now follows from the lower bound $n>n_0(\eps)$ in (\ref{Eq:ManyPoints}), and our choice (\ref{Eq:DefEps0}) of $\eps_0$.
\end{proof}


\begin{figure}[htb]
    \begin{center}
        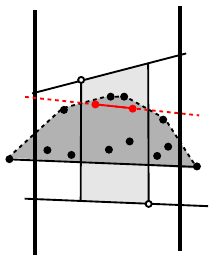\hspace{2cm}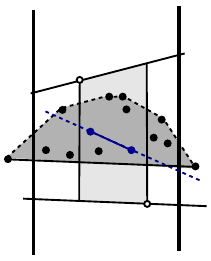
        \caption{\small Left: The short edge $uv$ is good for $K$ because all the points of $P_K^+\cup\{p,q\}$ that lie outside $\mu$ are to the same side of $L_{u,v}$. Right: The short edge $uv$ is bad for $K$.}
        \label{Fig:GoodEdge}
    \end{center}
\end{figure}

\paragraph{Definition.} We say that an edge $uv\in {P_\tau\choose 2}$ is {\it short} if its endpoints lie in the same cell $\mu\in \Sigma_\tau$. 

\medskip
Notice that the set $K^+$ contains 
$
{k_\mu\choose 2}=\Omega(\hat{\eps}^2n^2)
$
short edges within every full cell $\mu\in \Sigma_K$.
Let $uv$ be such a short edge whose endpoints belong to $P_K(\mu)$, for some cell $\mu$ of $\Sigma_K$.
We say that $uv$ is {\it  good} for $K^+$ if all the points of $P_K^+\cup\{p,q\}$ outside $\mu$ lie to the same side of the line $L_{u,v}$, and otherwise we say that $uv$ is {\it bad} for $K^+$; see Figure \ref{Fig:GoodEdge}.

\medskip
Informally, the good edges span lines that are nearly tangent to $K$.\footnote{We emphasize that the definition of a short edge is independent of $K$ whereas the notion of a good edge assumes both the underlying convex set $K$, and the principal edge $pq\in \Pi_K$ which crosses $\tau$ transversally.}
In particular, for every good edge $uv$ the corresponding line $L_{u,v}$ must miss the principal edge $pq$. Since $uv$ lies above $pq$, the edges $uv$ and $pq$ are boundary edges of a convex quadrilateral.

\begin{proposition}\label{Prop:GoodEdge}
\begin{enumerate}[label=(\roman*)]

\item Let $u_1v_1,u_2v_2,\ldots,\ldots u_kv_k$ be good edges with respect to $K$ so that no two of these edges lie in the same cell of $\Sigma_K$. Then the $k+1$ edges of $\{u_jv_j\mid 1\leq j\leq k\}\cup\{pq\}$ lie on the boundary of the same convex $(2k+2)$-gon; see Figure \ref{Fig:GoodProp}.

\item Let $\mu\in \Sigma_{K}$ be a full cell. Then the points of $P_K(\mu)$ determine at least $\displaystyle\frac{3}{4}{k_\mu\choose 2}$ good edges. 

\end{enumerate}
\end{proposition}

\begin{figure}
    \begin{center}
        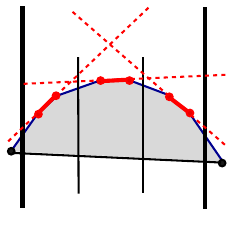
        \caption{\small Proposition \ref{Prop:GoodEdge}  (i) -- The good edges $u_jv_j$ with supporting lines $L_{u_j,v_j}$ are depicted. Since these edges lie in distinct cells, they bound a convex polygon (together with the principal edge $pq$).}
        \label{Fig:GoodProp}
    \end{center}
\end{figure}

\begin{figure}
    \begin{center}
      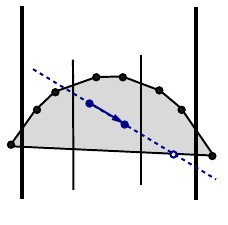\hspace{1.2cm}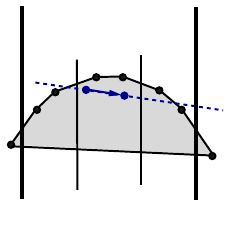\hspace{1.2cm}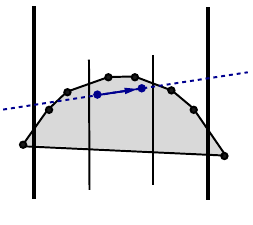
        \caption{\small Proof of Proposition \ref{Prop:GoodEdge} (ii). The bad edges of $E_1$, $E_2$, and $E_3$ are depicted (resp., left, center, and right). Notice that every edge $uv\in E_1$ is directed towards the intersection $L_{u,v}\cap pq$, whereas the edges of $E_2\cup E_3$ are directed rightwards.}
        \label{Fig:Bad123}
    \end{center}
\end{figure}

\begin{figure}
    \begin{center}
        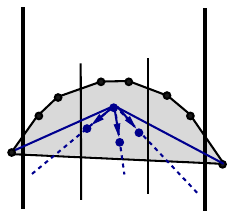\hspace{2cm}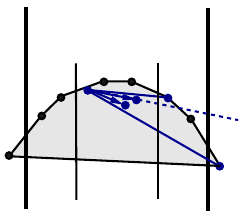
        \caption{\small Proof of Proposition \ref{Prop:GoodEdge} (ii). The point $u$ and its outgoing bad edges of $E_1$ and $E_2$ (resp., left and right). In each case, the other endpoints of the outgoing edges lie inside a triangle $T\subset K$ with apex $u$.}
        \label{Fig:GoodPropBound}
    \end{center}
\end{figure}

\begin{proof}
The first part readily follows from the definition of a good edge. To see the second part, we consider the set $E_{bad}$ of all the bad edges that are spanned by the points of $P_K(\mu)$. 
To bound its cardinality, we represent $E_{bad}$ as the union of the following subsets:
\begin{itemize}
\item $E_1$ consists of all the bad edges $uv$ whose supporting lines $L_{u,v}$ cross the principal edge $pq$ of $K$; see Figure \ref{Fig:Bad123} (left). 

\item $E_2$ (resp., $E_3$) consists of all the bad edges $uv\in E_{bad}\setminus E_1$ for which there is a point $w\in P_K^+$ that lies in a cell $\mu'\in \Sigma_\tau\setminus\{\mu\}$ crossed by $pq$ to the right (resp., left) of $\mu\cap pq$, so that $w$ is separated by $L_{u,v}$ from $pq$. See Figure \ref{Fig:Bad123} (center and right).

\end{itemize}

\medskip
Provided that $\hat{C}<1/120$, it suffices to show that each of these sets $E_1,E_2,E_3$ has cardinality at most $10\hat{C}{k_\mu\choose 2}$. (Notice that $E_2$ and $E_3$ may overlap, and for every edge $uv\in E_2\cup E_3$ the respective line $L_{u,v}$ misses $pq$.)

\medskip
\noindent{\it Bounding $|E_1|$.} Assume for a contradiction that $|E_1|\geq 10\hat{C}{k_\mu\choose 2}$. We direct every edge $uv\in E_1$ from $u$ to $v$ if the intersection of $L_{u,v}$ with the principal edge $pq$ is closer to $v$ than to $u$ (and otherwise we direct the edge from $v$ to $u$). By the pigeonhole principle, and since $k_\mu\geq 100$, there is a vertex $u\in P_K(\mu)$ whose out-degree is at least $\hat{C}\hat{\eps}n$. Hence, the triangle $T=\triangle pqu\subset K^+$ contains at least $\hat{C}\hat{\eps}n$ points of $P$, so $K$ must have been previously pierced by the auxiliary net $Q^\triangle\left(P,\hat{C}\hat{\eps}\right)$, and subsequently removed from $\K_{\tau}$ and $\K$. See Figure \ref{Fig:GoodPropBound} (left).

\medskip
\noindent{\it Bounding $|E_2|$ and $|E_3|$.} Since the definitions of $E_2$ and $E_3$ are fully symmetrical (up to reversing the $x$-axis), we bound only the cardinality of the former set. We direct every edge $uv\in E_2$ rightwards; see Figure \ref{Fig:Bad123} (center).\footnote{Specifically, if this edge is directed from $u$ to $v$ then $pu$ and $vq$ are edges of the convex quadrilateral $\conv(p,q,u,v)$.}

Once again, we assume for a contradiction that $|E_2|\geq 10\hat{C}{k_\mu\choose 2}$, so there is a vertex $u$ whose out-degree $d(u)$ is at least $\hat{C}\hat{\eps}n$. 
As in the previous case, we will find a triangle $T\subset K$ which contains all the $d(u)\geq \hat{C}\hat{\eps}n$ endpoints of the edges of $E_2$ that emanate from $u$; see Figure \ref{Fig:GoodPropBound} (right).

Indeed, let $uv_1,uv_2,\ldots,uv_{d(u)}=uv^*$ be the counterclockwise sequence of all the outgoing edges of $u$ in $E_2$ (so that the occupied sector of $\reals^2$ to the right of $u$ does not contain any of the points $p,q$). 
By the definition of $E_2\subset E\setminus E_1$, we can choose a point $w$ in a cell $\mu' \in \Sigma_K$ that is separated by $L_{u,v^*}$ from the edge $pq$, and is crossed by $pq$ to the right of $\mu\cap pq$.
Our analysis is assisted by the following property.

\begin{claim}\label{Claim:Triangle}
There is a line $L'$ that crosses both edges $uw$ and $uq$ and so that the entire segment $\triangle uqw\cap L$ lies outside the interior of $\mu$.
\end{claim}
\begin{proof}[Proof of Claim \ref{Claim:Triangle}.]
If $\mu$ and $\mu'$ are separated by a line $L'\in \R_1$, then this line must also cross $uq$; see Figure \ref{Fig:Triangle}. Indeed, the principal edge $pq$ meets $\mu$ and $\mu'$ in this horizontal order, so $L'$ crosses $pq$ in-between the intersections $pq\cap \mu$ and $pq\cap \mu'$. Thus, $q$ must lie to the same side of $L'$ as $\mu'$.

\begin{figure}[htb]
    \begin{center}
        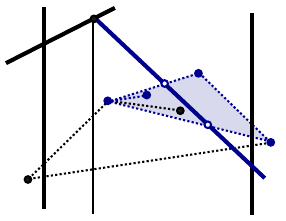
        \caption{\small Proof of Claim \ref{Claim:Triangle}. If $u$ and $w$ are separated by a line $L'$ of $\R_1$, this line must also cross the edge $uq$, and the intersection $\triangle uwq\cap L'$ must lie outside the interior of $\mu$.}
        \label{Fig:Triangle}
    \end{center}
\end{figure}

On the other hand, if $u$ and $w$ lie within the same cell $\rho$ in $\A(\R_1)$, then their respective cells $\mu\subset\rho$ and $\mu'\subset \rho$ in $\Sigma(r_1)$ must be separated by a vertical wall. Since $q$ lies to the right of $\tau$, and $pq$ crosses $\mu'$ to the right of its intersection with $\mu$, both $uw$ and $uq$ must cross the vertical line $L'$ which supports that wall (as depicted in Figure \ref{Fig:GoodPropBound} (right)).

\smallskip
In either case, the intersection $\triangle uwq\cap L'$ lies outside the interior of $\mu$ by the definition of $\A(\R_1)$ and $\Sigma(r_1)$.
\end{proof}

Let $a$ and $b$ be the respective $L'$-intercepts of $uw$ and $uq$ as depicted in Figure \ref{Fig:Triangle}.
Claim \ref{Claim:Triangle}, along with the convexity of $\mu$, imply that the triangle $T:=\triangle uab$ indeed contains the $d(u)\geq \hat{C}\hat{\eps}n$ points $v_1,\ldots,v_{d(u)}$ within $\mu$. As  before, this is contrary to the assumption that $K$ is missed by the strong $(\hat{C}\hat{\eps})$-net $Q^\triangle\left(P,\hat{C}\hat{\eps}\right)$. This contradiction completes the proof of Proposition \ref{Prop:GoodEdge}.
\end{proof}


\paragraph{Definition.} Let $\mu$ be a full cell of $\Sigma_K$. We orient every good edge within $\mu$ from the left to the right. We say that a point $u\in P_K(\mu)$ is {\it good} if it is adjacent to at least $k_\mu/10$ outgoing good edges (where, as before, $k_\mu$ denotes $|P_K(\mu)|$).

\medskip
By combining Proposition \ref{Prop:GoodEdge} (ii) with Proposition \ref{Eq:ManyFullCells}, and recalling that every full cell $\mu\in \Sigma_K$ satisfies $k_\mu\geq 100$, we obtain the following property:

\begin{proposition}\label{Prop:GoodPoints}
Every full cell $\mu\in \Sigma_K$ contains at least $k_\mu/4$ good points of $P_K^+$, for a total of at least $\eps_0 n/20$ such points.
\end{proposition}

\paragraph{Definition.} 
Let $u\in P^+_K$ be a good point. The {\it characteristic wedge $\W_K(u)$ at $u$} is the smallest planar wedge with apex $u$ that lies in the hyperplane to the right of $u$, and contains $uq$ along with all the outgoing good edges $uv$ of $u$ (but does not contain $up$); see Figure \ref{Fig:ABDeep} (left). Note that $\W_K(u)$ lies entirely in the halfplane to the right of $u$. 

\medskip
Denote 

$$
D(u):=|\left(P_\tau\cap \W_K(u)\right)\setminus \{u\}|.
$$ 
\noindent That is, $D(u)$ the number of the edges in $uw\in {P_\tau\choose 2}$ that are adjacent to $u$ and lie within the triangle $\W_K(u)\cap \tau$.\footnote{Recall that $P_\tau$ denotes the set $P\cap \tau$. Notice that many of these points $w\in (P_\tau\cap \W_K(u))\setminus \{u\}$, which contribute to the count $D(u)$, may not belong to $P_K$ or even to $P\cap K$.}
Since the characteristic wedge $\W_K(u)$ encompasses all the outgoing good edges of $u\in P_K(\mu)$, we trivially have $D(u)\geq k_\mu/10\geq \hat{\eps}n/10$, and $D(u)$ can be much larger than $\eps n$ due to the additional points of $P_\tau\setminus P_K$ that potentially lie within $\W_K(u)\cap \tau$.

\begin{figure}[htb]
    \begin{center}
        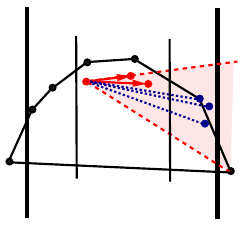\hspace{2cm}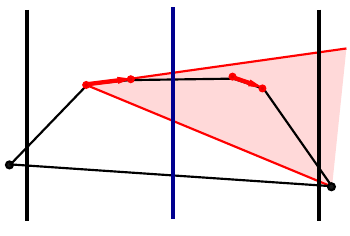
        \caption{\small Left: The characteristic wedge $\W_K(u)$ to the right of $u$ encompasses $uq$ and all the outgoing good edges of $u$. $D(u)$ denotes the overall number of the edges in ${P_\tau\choose 2}$ that are adjacent to $u$ and lie within the triangle $\W_K(u)\cap \tau$. Right: For each set $K$ of type $i$ we use the principal vertical line $L_K\in \Y(P,s_0)$ to split the subset $P_K(i)$ of the good points of type $i$ into subsets $A_K$ and $B_K$, of cardinality at least $\eps_1n/4$ each. For each point $u\in A_K$, the characteristic wedge $\W_K(u)$ contains all the points of $B_K$.}
        \label{Fig:ABDeep}
    \end{center}
\end{figure}

To interpolate between the two favourable scenarios sketched in Section \ref{Subsec:ProofOutline}, we subdivide the good points $u\in P^+_K$ into $O(\log 1/\eps)$ classes according to their respective degrees $D(u)$.

\medskip
\noindent{\bf Definition.} 
Let $i$ be an integer. We denote $\delta_i:=2^i\eps_1/4=2^i\eps_0/(160\log 1/\eps)$.
We say that a good point $u\in P^+_K$ is of {\it type $i$} if $\delta_i n \leq D(u)<\delta_{i+1}n$.
We use $P_K(i)$ to denote the (possibly empty) subset of all the good points of $i$-type in $P^+_K$. 

\begin{proposition}\label{Prop:TypeSet}
There is an integer  $\lfloor \log (2\hat{\eps}/5\eps_1)\rfloor\leq i\leq \log 1/\eps_1$
so that $|P_K(i)|\geq \eps_1 n$.
\end{proposition}

\begin{proof}
Since the degrees $D(u)$ of the good points $u\in P_K^+$ satisfy $\hat{\eps}n/10\leq D(u)\leq n/\lceil(r_0+1)\rceil$,
each point $u\in P_K^+$ is of some type $i$ in the asserted range, where the second inequality $i\leq \log 1/\eps_1$ uses (\ref{Eq:SmallConstantEps}) (which holds for all $\eps<\tilde{\eps}$). Combining (\ref{Eq:SmallConstantEps}) with the definition of $\eps_1$ and $\hat{\eps}$ in the beginning of Section \ref{Subsec:3}, we conclude that there exist at most $2\log 1/\eps$ types with $|P_K(i)|>0$. By Proposition \ref{Prop:GoodPoints} there exist at least $\eps_0n/20$ good points in $P_K^+$, and each of them belongs to one of these types $i$. Hence, by the pigeonhole principle, there must be an integer $i$ so that $|P_K(i)|\geq \eps_0 n/(40\log 1/\eps)=\eps_1 n$.
\end{proof}

\medskip
\noindent{\bf Definition.} We say that a convex set $K\in \K$ is {\it of type $i$} if $i:=\min\{j\in \mathbb{Z}\mid |P_K(j)|\geq \eps_1n \}$. According to Proposition \ref{Prop:TypeSet}, each convex set is of exactly one well-defined type $i\in \mathbb{Z}$. 

Let $K\in \K$ be a convex set of type $i$. Since $K$ is $(\eps_1/4)$-spread in the secondary slab decomposition $\Lambda(P,s_0)$, there must be a line $L_K\in \Y(P,s_0)$, within the principal middle slab $\tau$ of $K$, so that at least $\eps_1 n/4$ good points in $P_K(i)$ lie to each side of $L_K$. In what follows, we refer to $L_K$ as the {\it principal vertical line} of $K$. 

\medskip
\noindent{\bf Remark.} For the success of our strategy sketched in Section \ref{Subsec:ProofOutline}, it is quintessential that the (remaining) convex sets $K\in \K$ can be split using at most $s_0$ principal vertical lines $L_K\in \Y\left(P,s_0\right)$.

\medskip
\noindent{\bf Definition.} Let $K\in \K$ be a convex set of type $i$. We use $A_K$ (resp., $B_K$) to denote the subset of the good points in $P_K(i)$ that lie to the left (resp., right) of $L_K$; see Figure \ref{Fig:ABDeep} (right). Note that $|A_K|,|B_K|\geq \eps_1 n/4$.

\begin{proposition}\label{Prop:GoodWedge}
Let $K\in \K$ be a convex set, and $u$ be a good point in $A_K$. Then the characteristic wedge $\W_K(u)$ at $u$ contains at least $k_\mu/10\geq\hat{\eps}n/10$ outgoing good edges $uv$, and all the points of $B_K$. 
\end{proposition}

Since we have $|B_K|\geq \eps_1 n/4$ for every remaining convex set $K\in \K$ of type $i$, the proposition implies the following property.

\begin{corollary}\label{Eq:PositiveTypeSet}
Every remaining set $K\in \K$ has type $0\leq i\leq \log 1/\eps_1$.
\end{corollary}

\begin{proof}[Proof of Proposition \ref{Prop:GoodWedge}]
The desired number of the good edges in the characteristic wedge $\W_K(u)$ at any point $u\in A_K$ follows from the construction of $\W_K(u)$ (and because each point in $A_K$ is good and, therefore, satisfies, $D(u)\geq k_\mu/10$).

To show that $\W_K(u)$ also contains all the points of $B_K$, let $\mu\in \Sigma_K$ be the full cell that contains $u$. Since every point $u'\in B_K$ lies in a cell $\mu'\in \Sigma_K$ to the right of $\mu$ and $L_K$, the desired property follows from the first part of Proposition \ref{Prop:GoodEdge}. Indeed, since we have $\min\{D(u),D(u')\}\geq \min\{k_\mu/10,k_{\mu'}/10\}\geq 10$, good edges $uv$ and $u'v'$ can be chosen within, respectively, $\mu$ and $\mu'$ so that the vertices $u,v,u',v'$ appear in this horizontal order. Since $p,u,v,u',v',q$ form a convex chain by Proposition \ref{Prop:GoodEdge}, and the characteristic wedge $\W_K(u)$ contains the edges $uv$ and $uq$, it must also contain $u'$, as depicted in Figure \ref{Fig:ABDeep} (right).
\end{proof}

To pierce the remaining sets $K\in \K_\tau$ for each $\tau\in \Lambda(P,r_0)$, we combine the following two properties whose somewhat technical proofs are relegated to Section \ref{Subsec:ProofsLemmata}. 

\begin{lemma}\label{Lemma:Thick}
Let $\tau$ be a slab in $\Lambda(P,r_0)$, and let $K\in \K_\tau$ be a convex set of type $0\leq i\leq \log 1/\eps_1$. For each $u\in \A_K$, its respective characteristic wedge $\W_K(u)$ contains $\Omega(\delta_i n)$ edges of ${P_\tau\choose 2}$ that are adjacent to $u$ and cross the principal vertical line $L_K$ of $K$ within the interval $K\cap L_K$ (for a total of $\Omega(\eps_1 \delta_i n^2)$ such edges that cross $K\cap L_K$). See Figure \ref{Fig:Thick}.
\end{lemma}

\begin{figure}[htb]
    \begin{center}
        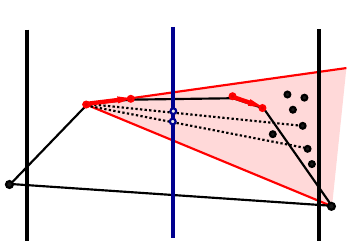
        \caption{\small Lemma \ref{Lemma:Thick} -- a schematic illustration. For each $u\in \A_K$, its characteristic wedge $\W_K(u)$ contains $\Omega(\delta_i n)$ edges that are adjacent to $u$ and cross the interval $K\cap L_K$.}
        \label{Fig:Thick}
    \end{center}
\end{figure}

\begin{lemma}\label{Lemma:SparseWedge}
For each $\tau\in \Lambda(P,r_0)$, and each $0\leq i\leq \log 1/\eps_1$, there is a subset $\Pi(\tau,i)\subset {P_\tau\choose 2}$ with the following properties:
 
\begin{enumerate}[label=\roman*.]
\item We have that
 \begin{equation}\label{Eq:SparseDelta}
|\Pi(\tau,i)|=O\left(\frac{\delta_i n^2}{{r_1^2}\hat{\eps}}\right).
\end{equation}

\item For each convex set $K\in \K_\tau$ of the type $i$, and each point $u\in A_K$, the set $\Pi(\tau,i)$ contains all the edges $uw\in {P_\tau\choose 2}$ that are adjacent to $u$ and lie within the characteristic wedge $\W_K(u)$ at $u$.
\end{enumerate}
\end{lemma}

Notice that the density of the graph $\Pi(\tau,i)$ is proportional to $\delta_i$, giving rise to the following tradeoff:

\begin{enumerate}
\item If $\delta_i$ exceeds $r_1\eps_1$ then we are in the first favourable scenario of Section \ref{Subsec:ProofOutline} -- combining Lemma \ref{Lemma:Thick} and Lemma \ref{Lemma:SparseWedge} (ii) for each $u\in A_K$ yields that the intercept $K\cap L_K$ is crossed by roughly $(r_1\eps_1 n) \cdot (\eps_1 n)\simeq r_1\eps^2n$ edges.

\item On the other hand, as $\delta_i$ approaches $\eps$, the set $\Pi(\tau,i)$ contains roughly $n^2/r_1$ edges, which gives rise to the second favourable scenario of Section \ref{Subsec:ProofOutline} (e.g., via Lemma \ref{Lemma:Sparse}, or through a direct application of Lemma \ref{Lemma:Thick}). 
\end{enumerate}

Our net $Q(\tau,i)$ for the convex sets $K\in \K_\tau$ of type $i$ interpolates between these two extreme cases.

\paragraph{The nets $Q_\tau$ and $Q_3$.}  For every $0\leq i\leq \log 1/\eps_1$, and every line $L\in \Y(P,s_0)$ within $\tau$, we add every intersection of $L$ with an edge of $\Pi(\tau,i)$ to the set $X(L,i)$. We then select every $\lceil C_3\eps_1\delta_i n^2\rceil$-th point of $X(L,i)$ into our net $Q(L,i)$, for a sufficiently small constant $C_3>0$. 

\medskip
We then define

$$
Q(\tau,i):=\bigcup\{Q(L,i)\mid L\in \Lambda(P,s_0),L\subset \tau\}.
$$

\noindent and 
$$
Q_\tau:=\bigcup_{0\leq i\leq \log 1/\eps_1} Q(\tau,i)
$$

\medskip
\noindent The complete net $Q_3$ at Stage 3 is then given by

$$
Q_3:=Q(P,s_0,\eps_1/4)\cup Q^\triangle\left(P,\hat{C}\hat{\eps}\right)\cup\bigcup_{\tau\in \Lambda(P,r_0)}Q_\tau.
$$

\begin{theorem}
With a suitably small constant $C_3>0$, the net $Q_3$ pierces every remaining convex set in $\K$ that is missed by the combination $Q_0\cup Q_1\cup Q_2$. Furthermore, we have that 

\begin{equation}\label{Eq:Stage3}
|Q_3|=O\left(s_0\cdot f\left(\eps\cdot\frac{s_0\cdot \sigma}{r_0\log 1/\eps}\right)+\frac{r^2_0r_1}{\eps  \sigma}\log^2 \frac{1}{\eps}+\frac{s_0 r_0^3\log^3 1/\eps}{\sigma^2 r_1\eps^2}\right).
\end{equation}

\end{theorem}

\begin{proof}
Using the definition of $Q_3$, we argue for each slab $\tau\in \Lambda(P,r_0)$ that the respective net $Q_\tau$ pierces all the sets $K\in\K_\tau$ that were missed by the previous nets $Q_0,Q_1,Q_2,Q(P,s_0,\eps_1/4)$ and $Q^\triangle\left(P,\hat{C}\hat{\eps}\right)$. 
It suffices to check, for all $\tau\in \Lambda(P,r_0)$, and all $0\leq i\leq \log 1/\eps_1$, that every convex set $K\in \K_\tau$ of type $i$ is pierced by one of the nets $Q(L,i)\subset Q_\tau$ whose vertical lines $L\in \Y(P,s_0)$ lie within $\tau$.

Indeed, according to Lemma \ref{Lemma:Thick}, every point $u\in A_K$ gives rise to $\Omega(\delta_i n)$ outgoing edges that cross the intercept $K\cap L_K$ with the principal vertical line $L_K\in \Y(P,s_0)$ which separates $A_K$ and $B_K$, for a total of $\Omega(\eps_1\delta_i n^2)$ such edges. Hence, choosing a small enough constant $C_3>0$ guarantees that $K$ is pierced by $Q(L_K,i)$.

\medskip
For every type $1\leq i\leq 1/\eps_1$, and every line $L\in \Y(P,s_0)$ within $\tau$, the cardinality of $Q(L,i)$ is bounded by
$$
O\left(\frac{|X(L,i)|}{\delta_i{\eps_1}n^2}\right)=O\left(\frac{|\Pi(\tau,i)|}{\delta_i{\eps_1}n^2}\right)=O\left(\frac{\delta_in^2}{r_1^2\hat{\eps}\delta_i \eps_1n^2}\right)=O\left(\frac{r_0\log r_1\log 1/\eps}{r_1\eps^2_0}\right).
$$

\noindent where the second equality uses the bound of Lemma \ref{Lemma:SparseWedge} (ii), and the third one uses the definitions of $\eps_1$ and $\hat{\eps}$.

\medskip
Recall that $\Lambda(P,s_0)$ is a refinement of $\Lambda(P,r_0)$, every slab $\tau\in \Lambda(P,r_0)$ contains $O(s_0/r_0)$ lines of $\Y(P,s_0)$. Using this and the definition (\ref{Eq:DefEps0}), we can bound the cardinality of $Q_\tau$ by 
$$
O\left(\frac{s_0}{r_0}\log(1/\eps) \cdot \frac{r_0\log r_1\log 1/\eps}{r_1\eps^2_0}\right)=O\left(\frac{s_0 r^2_0\log^3 1/\eps}{\sigma^2 r_1\eps^2}\right).
$$

Repeating this bound for each slab $\tau\in \Lambda(P,r_0)$ and combining it with the prior bounds (\ref{Eq:SpreadS}) and (\ref{Eq:TriangleNet}) on the cardinalities of the nets $Q(P,s_0,\eps_1/4)$ and $Q^\triangle\left(P,\hat{C}\hat{\eps}\right)$, we conclude that
the overall cardinality of $Q_3$ indeed satisfies the bound (\ref{Eq:Stage3}).
\end{proof}

\noindent{\bf Remark.} Since $s_0$ and $r_0$ are very small (albeit, fixed) positive power of $1/\eps$ that satisfy $s_0\gg r_0\gg 1/\sigma$, the recursive term in (\ref{Eq:Stage3}) is again near-linear in $1/\eps$. Furthermore, the two non-recursive terms sum up to roughly $r_1/\eps+1/(r_1\eps^2)$, so choosing $r_1=\Theta(\sqrt{1/\eps})$ renders them close to $1/\eps^{3/2}$.

\medskip
In Section \ref{Subsec:WrapUp} we combine (\ref{Eq:Stage3}) with the bounds on the sizes of the auxiliary nets $Q_0,Q_1$, and $Q_2$ of the previous Stages 0 -- 2 to derive a recurrence for $f(\eps)$ whose solution is close to $1/\eps^{3/2}$.

\subsection{Proofs of Lemmas \ref{Lemma:Thick} and \ref{Lemma:SparseWedge}}\label{Subsec:ProofsLemmata}

\paragraph{Proof of Lemma \ref{Lemma:Thick}.}
Refer to Figure \ref{Fig:EarsG}.
Fix a point $u\in A_K$, and let $uv$ be the good edge that delimits from above its characteristic wedge $\W_K(u)$. (In other words, $uv$ attains the largest slope among the good edges that emanate from $u$ to the right.)

By Proposition \ref{Prop:GoodWedge}, the wedge $\W_K(u)$ contains all the points of $B_K$.
We fix any of these points $u'\in B_K$ together with the edge $u'v'$ which delimits from above the respective wedge $\W_K(u')$. 
Since the point $u'$ too is of type $i$, the characteristic wedge $\W_K(u')$ contains at least $\delta_i n$ points $w\in P_\tau$. Since the edges $uv$ and $u'v'$ are good, Proposition \ref{Prop:GoodEdge} implies that the three edges $uv,u'v'$ and $pq$ form a convex 6-gon $G$. 
It, therefore, suffices to show that all the resulting edges $uw$ cross the intercept $G\cap L_K\subset K\cap L_K$, where $L_K$ denotes the principal vertical line of $K$.

Let $L_0\in \Y(P,r_0)$ (resp., $L_1\in \Y(P,r_0)$) be the line that supports $\tau$ from the left (resp., right). Notice that $L_0$ is crossed by the edges $pq$ and $pu$, and $L_1$ is crossed by the edges $pq$ and $v'q$, and none of the remaining edges $uv$, $vu'$, and $u'v'$, of $G$ crosses $L_0$ or $L_1$.  
Thus, the intersection $G_\tau:=G\cap \tau$ is a convex $8$-gon. 
The claim now follows since (1) $\W_K(u')$ is separated from $u$ by $L_K$, and (2) every point $w\in \W_K(u)\cap P_\tau$ lies either inside $G_\tau\subset K$, or in the triangular ``ear" that is adjacent to the edge $v'q\cap \tau$ of $G_\tau$ and delimited by $L_1$ and $L_{u'v'}$.
$\Box$

\begin{figure}[htb]
    \begin{center}
      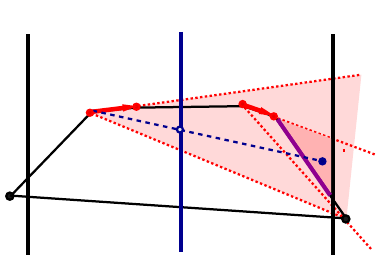
        \caption{\small Proof of Lemma \ref{Lemma:Thick}. The convex $6$-gon $G=\conv(p,q,u,v,u',v')$ is depicted. Every point $w\in \W_K(u')$ is separated from $u$ by $L_K$. It lies either inside $K$, or in the triangular ``ear" that is adjacent to the edge $v'q\cap \tau$  of $G_\tau=G\cap \tau$ and delimited by $L_1$ and $L_{u',v'}$.}
        \label{Fig:EarsG}
    \end{center}
\end{figure}

\paragraph{Proof of Lemma \ref{Lemma:SparseWedge}.}
We first describe the sparse subgraph $\Pi(\tau,i)\subset {P_\tau\choose 2}$ for all $\tau\in \Lambda(P,r_0)$ and $0\leq i\leq \log 1/\eps_1$.
\begin{figure}[htb]
    \begin{center}
        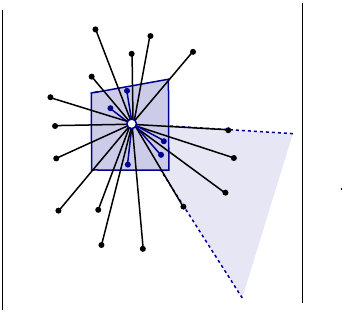       
         \caption{\small Proof of Lemma \ref{Lemma:SparseWedge} -- defining the sparse graph $\Pi(\tau,i)\subset {P_\tau\choose 2}$. We define $z=O(1/\delta_i)$  sectors $\W_j(u)$ with apex $p$. In each sector, the number of the edges of ${P_\tau\choose 2}$ that are adjacent to $p$ ranges between $2\lceil 2\delta_i n\rceil$ and $3\lceil 2\delta_i n\rceil$. We add the edges of $\W_j(u)$ to $\Pi(\tau,i)$ only if this sector is {\it rich} and encompasses at least $\hat{\eps}n/10$ short edges.}\label{Fig:Sparse}
    \end{center}
\end{figure}

\medskip
\noindent{\it The graph $\Pi(\tau,i)$.} Denote $P_\tau:=P\cap \tau$ and $n_\tau:=|P\cap \tau|$.
For each $u\in P_\tau$ which lies in some cell $\mu\in \Sigma_{\tau}$ we partition the $n_\tau-1$ adjacent edges $uv_1,\ldots,uv_{n_\tau-1}\in {P_\tau\choose 2}$ (which appear in this clockwise order around $u$) into $z=O(1/\delta_i)$ blocks $\E_j$, for $0\leq j\leq z-1$, so that every block but the last one contains $\lceil 2\delta_i n\rceil$ edges, and the last block contains at most $\lceil 2\delta_i n\rceil$ edges. 
If $z\geq 4$, we define $z$ canonical sectors with apex $u$, where each sector $\W_j(u)$
encompasses three consecutive blocks $\E_j,\E_{j+1},\E_{j+3}$ of edges, and the indexing is modulo $z$. See Figure \ref{Fig:Sparse}. Otherwise (i.e., if $\lceil 2\delta_i n\rceil >(n-1)/3$), we define only one sector $\W_0(u)=\reals^2$.

Notice that, given that $z\geq 4$, the neighboring sectors overlap, each sector $\W_j(u)$ satisfies $2\lceil2\delta_i n\rceil\leq |(\W_j(u)\cap P_\tau)\setminus \{u\}|\leq 3\lceil 2\delta_i n\rceil$, and each edge $uv$ lies in exactly three of the sectors of $u$. 

\medskip
We say that the sector $\W_j(u)$ is {\it rich} if $|(\W_j(u)\cap P_\mu)\setminus\{u\}|\geq \hat{\eps} n/10$. In other words, the sector $\W_j(u)$ must contain at least $\hat{\eps} n/10$ short edges $uv$.

\medskip
We add to $\Pi(\tau,i)$ every edge $uv$ that lies a rich sector of at least one of its endpoints $u$ or $v$.

\bigskip
\noindent{\it Analysis.} To see the first property of $\Pi(\tau,i)$, it is sufficient to show that any point $u\in P_\tau$ contributes $\displaystyle O\left(\frac{\delta_i n}{{r_1^2}\hat{\eps}}\right)$ edges to the set $\Pi(\tau,i)$.

Indeed, recall that for each cell $\mu\in \Sigma_{\tau}$ we have that $n_\mu=|P_\mu|\leq n/r_1^2$. Therefore, the pigeonhole principle implies that for each $u\in P_\mu$ there can be only $\displaystyle O\left(\frac{n}{r_1^2\hat{\eps}n}\right)=\displaystyle O\left(\frac{1}{r_1^2\hat{\eps}}\right)$ rich sectors $\W_j(u)$, which satisfy $|\W_j(u)\cap P_\mu\setminus\{u\}|\geq \hat{\eps} n/10$, and any such sector contributes $O(\delta_i n)$ edges to $\Pi(\tau,i)$. 

\medskip
For the second property, we recall that, for every good point $u\in A_K$ that lies in some full cell $\mu\in \Sigma_K$, the respective wedge $\W_K(u)$ contains at most $2\delta_in$ outgoing edges $uw$ within $\tau$ and, therefore, is contained in (at least) one of the sectors $\W_j(u)$.
Proposition \ref{Prop:GoodWedge} now implies that this sector $\W_j(u)$ is rich, for it contains at least $k_\mu/10\geq \hat{\eps}n/10$ outgoing short edges $uv$ (where, as before, $k_\mu$ denotes $|P_K(\mu)|$). $\Box$


\subsection{The final recurrence} \label{Subsec:WrapUp}
In this section we derive the complete recurrence of the general form (\ref{Eq:GenRecurrence}) for the quantity $f(\eps,\lambda,\sigma)$. As mentioned in Section \ref{Subsec:RecursiveFramewk}, this will yield a simpler recurrence of the form (\ref{Eq:GenRecurrenceSimple}) for the quantity $f_2(\eps)=f(\eps)$. This ultimately solves to $\displaystyle f(\eps)=O\left(\frac{1}{\eps^{3/2+\gamma}}\right)$, where $\gamma>0$ is an arbitrary small constant that has been fixed in the beginning of this proof.

As mentioned in the beginning of this section, our recurrence will involve the auxiliary parameters $r_0\ll s_0$ which are very small (albeit fixed) degrees of $1/\eps$ that depend on $\gamma$ and satisfy $r_0=\Theta(s_0^\eta)$ and $s_0=\Theta\left(1/\eps^\eta\right)$, where $\eta:=\gamma/100$.

As mentioned in Section \ref{Sec:Prelim}, we fix a suitably small constant $0<\tilde{\eps}<1$ that satisfies (\ref{Eq:SmallConstantEps}) and use the old bound $\displaystyle f(\eps)=O\left(\frac{1}{\tilde{\eps}^{2}}\right)=O(1)$ of Alon {\it et al.} \cite{AlonSelections}  whenever $\eps\geq \tilde{\eps}$. (The choice of $\tilde{\eps}$ will affect the multiplicative constant in the eventual asymptotic bound on $f(\eps)$.) Assume then that $\eps<\tilde{\eps}$.

\paragraph{Bounding $f(\eps,\lambda,\sigma)$.} To obtain a bound of the general form (\ref{Eq:GenRecurrence}) for $f(\eps,\lambda,\sigma)$, where $\eps<\tilde{\eps}$ and $\lambda>\eps$, we fix a family $\K=\K(P,\Pi,\eps,\sigma)$ that satisfies $|\Pi|/{P\choose 2}\leq\lambda$, and bound the overall cardinality of the point transversal $Q$ for $\K$ that was constructed in Sections \ref{Subsec:0} through \ref{Subsec:3}.  As explained in the beginning of Section \ref{Sec:Main}, we can also assume that the cardinality of $P$ is bounded from below by (\ref{Eq:ManyPoints}), so that $|P|=n>n_0(\eps)$.

\medskip
Combining (\ref{Eq:EarlyStages}) and (\ref{Eq:Stage3}) yields the following bound on the overall cardinality of our net $Q$:

$$
f\left(\eps,\frac{\lambda}{r_0},\frac{\sigma}{2}\right)+
$$

$$
+O\left(r_0\cdot f(\eps\cdot\sigma\cdot r_0)+\frac{r_0^2r_1}{\eps}+s_0\cdot f\left(\eps\cdot\frac{s_0\cdot \sigma}{r_0\log 1/\eps}\right)+\frac{r^2_0r_1}{\eps  \sigma}\log^2 \frac{1}{\eps}+\frac{s_0 r_0^3 \log^3 1/\eps}{\sigma^2 r_1\eps^2}\right).
$$

\noindent By substituting $r_1=\Theta\left(\sqrt{1/\eps}\right)$ and rearranging the terms, we conclude for all $\eps<\tilde{\epsilon}$ and $\lambda>\eps$ that

\begin{equation}\label{Eq:Q}
f(\eps,\lambda,\sigma)\leq f\left(\eps,\frac{\lambda}{r_0},\frac{\sigma}{2}\right)+O\left(\Psi(\eps,\sigma)\right),
\end{equation}

\noindent where 
$$
\Psi(\eps,\sigma):=s_0\cdot f\left(\eps\cdot s_0\cdot \frac{\sigma}{r_0\log 1/\eps}\right)+r_0\cdot f(\eps\cdot r_0\cdot \sigma)+\frac{r_0^2}{\eps^{3/2}}+\frac{r_0^2}{\sigma \eps^{3/2}}\log^2 \frac{1}{\eps}+\frac{s_0 r_0^3 \log^3 1/\eps}{\sigma^2 \eps^{3/2}}.
$$

\noindent Our choice of $r_0$ and $s_0$, in combination with (\ref{Eq:SmallConstantEps}), yields

$$
\Psi(\eps,\sigma)=O\left(s_0\cdot f\left(\eps\cdot s_0^{1-2\eta}\cdot \sigma\right)+r_0\cdot f(\eps\cdot r_0\cdot \sigma)+\frac{1}{\sigma^2\eps^{3/2+7\eta}}\right).
$$

\paragraph{The recurrence for $f(\eps)$.} We begin with $f(\eps)=f(\eps,1,1)$ and recursively apply the inequality (\ref{Eq:Q}) to the ``leading" term, which involves the density $\lambda$, while keeping the parameters $\eps,r_0,s_0$ and $r_1$ fixed. This recurrence in $\lambda$ bottoms out when the value of $\lambda$ falls below $\eps$. 
Since $r_0$ is a fixed (though very small) positive power of $1/\eps$, the inequality (\ref{Eq:Q}) is applied $J=O\left(\log_{r_0} 1/\eps\right)=O(1)$ times.
Hence, the value of the restriction threshold $\sigma$ in the $j$-th application is bounded from below by $1/2^{j-1}=\Theta(1)$.
Using the trivial property that $f(\eps,\lambda',\sigma')\geq f(\eps,\lambda,\sigma)$ and $\Psi(\eps,\sigma')\geq \Psi(\eps,\sigma)$ for all $0\leq \lambda\leq \lambda'\leq 1$ and $0<\sigma'\leq \sigma\leq 1$, and that $J=O(1)$, we conclude that

\begin{equation}\label{Eq:ReqIterate}
f(\eps)=f(\eps,1,1)\leq f\left(\eps,\eps,2^{-J}\right)+\sum_{i=1}^J\Psi\left(\eps,2^{-j+1}\right)=f\left(\eps,\eps,2^{-J}\right)+O\left(\Psi\left(\eps,2^{-J}\right)\right).
\end{equation}


\noindent Note that
\begin{equation}\label{Eq:BoundPsi}
\Psi\left(\eps,2^{-J}\right)=O\left(s_0\cdot f\left(\eps\cdot s_0^{1-2\eta}\right)+r_0\cdot f(\eps\cdot r_0)+\frac{1}{\eps^{3/2+7\eta}}\right).
\end{equation}

\medskip
\noindent To bound $f\left(\eps,\eps,2^{-J}\right)$, we invoke Lemma \ref{Lemma:Sparse} with $r:=r_0$, which yields

\begin{equation}\label{Eq:RecurrenceSparse}
f(\eps,\eps,2^{-J})=O\left(r_0\cdot f\left(\eps\cdot r_0\right)+\frac{r_0^2}{\eps}\right)=O\left(r_0 \cdot f\left(\eps\cdot r_0\right)+\frac{1}{\eps^{1+2\eta}}\right).
\end{equation}

\noindent Substituting (\ref{Eq:RecurrenceSparse}) and (\ref{Eq:BoundPsi}) into (\ref{Eq:ReqIterate}) readily gives

\begin{equation}\label{Eq:RecurrenceComplete}
f(\eps)=O\left(s_0\cdot f\left(\eps\cdot s_0^{1-2\eta}\right)+r_0\cdot f\left(\eps\cdot r_0\right)
+\frac{1}{\eps^{3/2+7\eta}}\right),
\end{equation}

\noindent where the implicit constants do not depend on the particular choice of the constant upper threshold $\tilde{\eps}$ for $\eps$ as long as the inequality (\ref{Eq:SmallConstantEps}) is satisfied. (As previously mentioned, we routinely omit the constant factors within the recursive terms of the form $f(\eps\cdot h)$.) A suitably small choice of $\tilde{\eps}>0$ guarantees that $\eps$ indeed increases by a factor of at least $2$ with each invocation of the recurrence.

\bigskip
\noindent{\bf Solving the recurrence for $f(\eps)$.} This last recurrence  (\ref{Eq:RecurrenceComplete}) bottoms out when $\eps\geq \tilde{\eps}$, in which case we have $f(\eps)=O\left(1/\tilde{\eps}^2\right)=O(1)$. Since $\eta=\gamma/100$, fixing a suitably small constant threshold $\tilde{\eps}>0$, and following the standard induction argument which applies to recurrences of this type (see, e.g., \cite{MatWag04}, and also \cite{Envelopes3D,EnvelopesHigh} and \cite[Section 7.3.2]{SA}), yields

\begin{equation}\label{Eq:ExplicitInequalityA}
f(\eps)\leq \frac{F}{{\eps}^{3/2+\gamma}},
\end{equation}
for all $\eps>0$, where $F\geq f\left(\tilde{\eps}\right)=O\left(1/\tilde{\eps}^2\right)$ \cite{AlonSelections} is a suitably large constant.

\smallskip
For the sake of completeness, we spell out the key details of this generic induction.\footnote{Since we trivially have $f(\eps)\leq f(\eps')$ for all $0<\eps/2\leq \eps'\leq \eps\leq 1$, it is sufficient to establish the asymptotic bound for the values $\eps$ of the form $\eps=1/2^j$, where $j\in \mathbb{N}$.} 
We first choose a constant threshold $\varepsilon(\gamma)$ so that the following conditions are satisfied for all $\eps\leq \varepsilon(\gamma)$: (i) the inequality (\ref{Eq:SmallConstantEps}) holds, and (ii) each of the recursive terms of the form $f(\eps')$ in (\ref{Eq:RecurrenceComplete})
involves $\eps'>2\eps$. 

Since the inequality (\ref{Eq:ExplicitInequalityA}) trivially holds for $\tilde{\eps}\leq \eps\leq 1$ whenever $F\geq f\left(\tilde{\eps}\right)$, it suffices to choose the constant threshold $\tilde{\eps}\leq \varepsilon(\gamma)$ that would facilitate the induction step for all the smaller values $\eps<\tilde{\eps}$.
To  this end, we plug the desired induction assumption (with so far unknown $F>0$) in (\ref{Eq:RecurrenceComplete}) so as to replace the two recursive terms $f\left(\eps\cdot s_0^{1-2\eta}\right)$ and $f\left(\eps\cdot r_0\right)$. This substitution readily yields

\begin{equation}
f(\eps)\leq \frac{F}{\eps^{3/2+\gamma}}\cdot \Phi(\eps),
\end{equation}

\noindent where 

\begin{equation}\label{Eq:Phi}
\Phi(\eps)\leq H_1\cdot\frac{s_0}{s_0^{(1-2\eta)(3/2+\gamma)}}+H_2\cdot\frac{r_0}{r_0^{3/2+\gamma}}+H_3\cdot\eps^{\gamma-7\eta}.
\end{equation}
Note that the positive constants $H_1,H_2$ and $H_3$ are determined by (\ref{Eq:RecurrenceComplete}), and do not depend on the particular choice of $\tilde{\eps}>\varepsilon(\gamma)$ or $F>0$. Using that $\eta=\gamma/100$, we at last choose the threshold $\tilde{\eps}\leq \varepsilon(\gamma)$ so that the right hand side of (\ref{Eq:Phi}) remains smaller than $1$ for all $\eps<\tilde{\eps}$. 
Invoking (\ref{Eq:RecurrenceComplete}), and using the induction assumption with this particular choice of $\tilde{\eps}$ and $F=f\left(\tilde{\eps}\right)$, indeed confirms the inequality (\ref{Eq:ExplicitInequalityA}) for all $\eps<\tilde{\eps}$. This concludes the proof of Theorem \ref{Thm:Main}. $\Box$

\section{Concluding remarks}\label{Sec:Conclude}
\begin{itemize}
\item Our analysis is largely inspired by the partition-based proof  \cite{ManyCells} of the Szemer\'{e}di-Trotter Theorem \cite{SzT} on the number of point-line incidences in the plane.  In the case at hand, narrow convex sets are viewed as abstract lines, so a non-trivial incidence bound implies that a typical point $u\in P$ is involved in $o(1/\eps)$ such canonical sets (which are naturally associated with the surrounding sectors $\W_j(u)$). This gives rise to a sparse restriction graph $\Pi$. 

Therefore, it is no surprise that our main decomposition $\Sigma(r_1)$ overly repeats the one that was used by Clarkson {\it et al.} \cite{ManyCells} in order to extend the Szemer\'{e}di-Trotter bound to more general settings (e.g., incidences between points and unit circles, and incidences between lines and certain cells in their arrangement). 

\item As previously mentioned, we did not seek to optimize the implicit constant factor in our asymptotic bound for $f(\eps)$. However, it is known to heavily depend on the choice of the constant $\gamma>0$, and it is at least $2^{\Omega\left(1/\gamma^2\right)}$ for recurrences of the form (\ref{Eq:RecurrenceComplete}); see the discussion in the end of  \cite[Section 7.3.2]{SA}. 

\item Our proof of Theorem \ref{Thm:Main} is fully constructive, and the resulting net includes  points of the following types:

\begin{enumerate}
\item The vertices of the decompositions $\Sigma(r_1)$ which arise in the various recursive instances.

\item 1-dimensional $\hat{\eps}$-nets within lines $L\in \Y(P,r_0)$, for $\hat{\eps}=\tilde{\Omega}(\eps^{3/2})$. In each net of this kind, the underlying point set is composed of the $L$-intercepts of the edges of ${P\choose 2}$. These edges typically belong to one of the sparser graphs $\Pi_{>t}$ (in Section \ref{Subsec:2}) or $\Pi(\tau,i)$ (in Section \ref{Subsec:3}).

\item 1-dimensional $\hat{\eps}$-nets within lines $L\in \Y(P,r_0)$, for $\hat{\eps}=\tilde{\Omega}(\eps^{3/2})$, where the underlying point sets are composed of the $L$-intercepts of the ``mixed" edges, which connect the vertices of $\Sigma(r_1)$ to the points of $P$.

\item 2-dimensional $\hat{\eps}$-nets of Theorem \ref{Thm:StrongNet} with respect to triangles in $\reals^2$.
\end{enumerate}

\item Our construction and its analysis combine classical elements of the 30-year old theory of linear arrangements in computational geometry (which generalize to any dimension) with a few ad-hoc arguments in $\reals^2$ (which do not immediately extend to higher dimensions).
A recent study \cite{STOC} of the author combines the ``incidences counting" strategy for the ``narrow" convex sets (which loosely resemble hyperplanes) with novel selection-type results \cite{AlonSelections} in order to show that the comparable bound $f_d(\eps)=O\left(1/\eps^{d-1/2+\gamma}\right)$, for any $\gamma>0$, holds in all dimensions $d\geq 3$. Its remarkable that, despite the superficial similarity with the bound of Theorem \ref{Thm:Main}, applying the higher-dimensional analysis in the plane comes short of re-establishing Theorem \ref{Thm:Main}.


The author conjectures that the actual asymptotic behaviour of the functions $f_d(\eps)$  in any dimension $d\geq 1$ is close to $1/\eps$, as is indeed the case for their ``strong" counterparts with respect to simply shaped objects in $\reals^d$ \cite{HW87}.
Its worth mentioning that our main argument in Section \ref{Subsec:3} exploits the delicate interplay between the two notions of $\eps$-nets which was partly explored by Mustafa and Ray \cite{MustafaRay} and, more recently, by Har-Peled and Jones \cite{NetShape}.

\item 

The recently improved analysis of the transversal numbers $C_d(p,q)$ that arise in the Hadwiger-Debrunner problem (see Section \ref{sec:intro}), due to Keller, Smorodinsky, and Tardos \cite{Shakhar}, implies that

\begin{equation}\label{Eq:Chaya}
C_d(p,q)\leq f_d\left(\Omega\left(p^{-1-\frac{d-1}{q-d}}\right)\right).
\end{equation}

Here, as before, $f_d(\eps)$ denotes the smallest possible number $f$ with the property that any $n$-point set $P\subset \reals^d$ admits a weak $\eps$-net of cardinality $f$ with respect to convex sets.

Plugging the result of Theorem \ref{Thm:Main} into 
(\ref{Eq:Chaya}) yields an improved bound in dimension $d=2$:

$$
C_2(p,q)=O\left(p^{(3/2+\gamma)\left(1+\frac{1}{q-2}\right)}\right)
$$
 
for any constant parameter $\gamma>0$, and $p$ larger than a certain constant threshold which depends on $\gamma$.

\item As the primary focus of this study is on the combinatorial aspects of weak $\eps$-nets, we did not seek to optimize the construction cost of our net $Q$.

A straightforward implementation of the recursive construction of $Q$ runs in time $\tilde{O}\left({{n^2}/{\sqrt{\eps}}}\right)$. The construction of $\Sigma(\R_1)$ from the sample $\R_1\subset \L(\Pi)$, the assignment of the points of $P$ to the trapezoidal cells, and the zones of the lines of $\L(\Pi)$, can all be performed using the standard textbook algorithms \cite{Mulmuley,SA}. Most of the running time is spent on explicitly maintaining the restriction graphs $\Pi$ along with the sparse graphs of Section \ref{Subsec:3}, and tracing the zones of the lines of $\L(\Pi)$ in $\Sigma(r_1)$. 

\end{itemize}

\paragraph{Acknowledgement.} The author would like to thank J\'{a}nos Pach, Micha Sharir and G\'{a}bor Tardos for their numerous invaluable comments on the early versions of this paper.
In particular, he is indebted to G\'{a}bor Tardos for pointing out that Proposition \ref{Prop:Narrow1} extends to vertical decompositions, which substantially simplified the subsequent exposition.

In addition, the author would like to thank an anonymous Journal of the ACM referee for suggestions which helped to improve the presentation.

\end{document}